\documentclass[11pt,a4paper]{article}
\usepackage{amssymb,amsmath,amsthm,bm}
\usepackage{graphicx}
\usepackage{geometry}
\usepackage{float}
\usepackage{tikz}
\graphicspath{{./figures/}}
\usepackage[maxbibnames=99]{biblatex}
\newtheorem*{teo*}            {Theorem}
\newtheorem{teo}            {Theorem}
\newtheorem{lema}     [teo]{Lemma}
\newtheorem{defin}[teo]{Definition}
\newtheorem{corollary} [teo]{Corollary}

\newtheorem{example} [teo]          {Example}
\newtheorem{obs} [teo]           {Remark}
\newtheorem{prop} [teo]       {Proposition}

\DeclareMathOperator{\pos}{pos}
\DeclareMathOperator{\Spec}{Spec}
\DeclareMathOperator{\ric}{Ric}
\DeclareMathOperator{\Ric}{Ric}

\title{Ricci curvature of Bruhat orders}
\author{Viola Siconolfi\ \footnotemark[1]}

\begin{document}
\maketitle
\footnotetext[1]{
Fakult\"at f\"ur Mathematik, Universit\"at Bielefeld, Germany.
\textit{Email address: }\texttt{vsiconolf@math.uni-bielefeld.de}.}

\begin{abstract}
We study the Ricci curvature of the  Hasse diagrams of the  Bruhat order of  finite irreducible Coxeter groups. For this purpose we compute the maximum degree of these graphs for types $B_n$ and $D_n$. The proof uses a new graph $\Gamma(\pi)$ defined for any element $\pi$ in the corresponding group.
\end{abstract}

\section{Introduction}
The study of discrete analogues of concepts from differential geometry is a wide spread research topic, that allows to work with the definition of curvature of graphs (\cite{bre1},\cite{bre8},\cite{bre10},\\
\cite{bre11},\cite{bre12},\cite{bre14}) with applications to heat and Laplacian operators (\cite{bre3},\cite{heat}),  isoperimetry (\cite{bre17},\cite{isop}) and other inequalities (\cite{bre2},\cite{bre4}, \cite{bre7},\cite{bre5},\cite{bre9},\cite{bre16}).

In particular, a definition of Ricci curvature for graphs was proposed in \cite{shmuck} and further studied in \cite{klart}. The study of Ricci curvature for graphs is still an elusive issue, among the difficulties in this framework there is the lack of examples and explicit computations that assist our intuition.

In \cite{sic},\cite{sic2} and \cite{sic3} the Ricci curvature  of Bruhat graphs of Coxeter groups and of the Hasse graphs of the weak orders of Coxeter and affine Weyl groups is studied.

%In this paper we consider an important family of graphs from Coxeter Theory, namely Bruhat order for finite irreducible groups, and we study a lower bound for the Ricci curvature of these objects. An indepth analysis for the Ricci curvature of graphs related to Coxeter Theory can be found in \cite{phd}. For example in Chapter $3$ of \cite{phd} we see that the Ricci curvature of the Bruhat graph of any finite Coxeter group is $2$ and in Chapter $4$ one can find the study of the weak order graphs related to the same groups. 
These results lead to applications such as a lower bound for the spectral gap of the graphs considered and isoperimetric inequalities.

%In this paper the process of obtaining bounds for the Ricci curvature of Bruhat graphs requires a detailed study of the structure of these groups, in particular for the infinite families $A_n$, $B_n$ and $D_n$, and brings to results about these groups that may be of interest regardless of the Ricci curvature.

In this work we consider the Ricci curvature of the Hasse graphs of the Bruhat order of finite irreducible Coxeter systems. This is a considerably harder problem that those studied by the same author on \cite{sic2} and \cite{sic3} because not all vertices are 'isomorphic' in these graphs.

Our approach is to apply to the Bruhat orders some results obtained in \cite{sic3} bound the Ricci curvature of a graph in terms of the degree of its vertices. 
%More precisely we obtain:
%\begin{teo*}
%Let $G$ be a triangle free graph then 
%\[
%\Ric(G)\geq 4-\max_{x,y\in \mathcal{E}(G)} \left( \frac{3d(x)+d(y)}{2} \right)
%\]

%\end{teo*}
%Where $\mathcal{E}(G)$ denotes the set of vertices of $G$ and $d(x)$ is the degree of a vertex $x$.
More precisely one needs to compute the maximal degree of the vertices of these graphs.
 Such a result is known  in type $A_n$, (see \cite{adinroich})

% was not in the literature for groups of type  $B_n$ and $D_n$ and its computation is a substantial part of this paper. The maximal degree of the strong Bruhat order in type $A_n$ is studied in \cite{adinroich}, while for some other finite exceptional Coxeter groups (e.g $E_6, H_3, F_4$) is a result of a SageMath computation (see Prop. \ref{H(W)}).

%The study of the maximal degree of Bruhat orders of type $B_n$ and $D_n$ leads to Theorems \ref{teobn} and \ref{degdn} 

A sustantial part od this paper is devoted to the proof of the corresponding result in types $B$ and $D$, more precisely we obtain the following theorems
 whose proofs appear in Subsections \ref{dmaxbn} and \ref{dmaxdn} :
\begin{teo*}
The maximum degree of the Hasse diagram of the strong Bruhat order in $B_n$ with $n\geq 5$ is 
$\lfloor \frac{n^2}{2}\rfloor+n-1.$
\end{teo*}
\begin{teo*}
The maximum degree of the Hasse diagram of the strong Bruhat order in $D_n$ is 
$\lfloor \frac{n^2}{2}\rfloor+n-1.$
\end{teo*}
 The main point of the proofs is the definition, for any element $\pi\in B_n$ (resp. $D_n$) of a graph $\Gamma(\pi)$ that describes the edge structure of the Hasse diagrams Bruhat order in a neighborhood of $\pi$. In particular the number of edges of $\Gamma(\pi)$ ($\Gamma(\mu)$) coincides with the number of elements that cover, or are covered by, $\pi$ in the Bruhat order . The main part of the proofs of these theorems \ref{teobn} and \ref{degdn} is the study of the maximal degree in graphs of type $\Gamma(\pi)$.

This work is subdivided into two parts, we now list the contents of each of these parts.\\
Section 2 is a preliminaries section. We introduce the Ricci curvature of a graph and briefly describe the historical motivation for such a definition. We then present some results related to the computation of the discrete Ricci curvature for graphs. We conclude this  part by recalling the main notions of Coxeter theory used in the rest of the paper.

Section 3 contains the new results of this work. We consider the  finite rreducible Coxeter groups and study the Ricci curvature of the Hsse graph of the Bruhat order. We start from the  dihedral case, the result follows from a particular property of the local structure of these graphs. We then study the Ricci curvature of the other irreducible Coxeter groups in terms of the maximal degree of the vertices of these graphs. This follows from known results ($A_n$), from a SageMath computation ($E_6,F_4,H_3$) and from new results ($B_n$, $D_n$). Subsections 3.1 and 3.2 are devoted to the proof in types $B_n$ and $D_n$ respectively.

\section{Preliminaries}\label{prelimin}

\subsection{Ricci curvature of a locally finite graph}\label{radici}
We recall here the definition and main facts about the discrete Ricci curvature of a locally finite graph, the main reference for this section is Subsection $1.1$ from \cite{klart}, the notation is slightly different.  

Let $G$ be a graph, 
we assume it to be undirected, with no loops and with no multiple edges (i.e. a simple graph); also we ask that it has no isolated vertices.
We denote by $\mathcal{V}(G)$ the set of vertices of $G$, by $\mathcal{E}(G)$ its edges and we take $\delta(x,y)$ as the function $\delta:\mathcal{V}(G)\times\mathcal{V}(G)\rightarrow \mathbb{N}\cup \{\infty\}$ that gives the distance between two vertices. For a fixed $x\in\mathcal{V}(G)$ and $i\in\mathbb{N}$ we define the set:
\[
B(i,x):= \{u\in\mathcal{V}(G)|\delta(x,u)=i\}; 
\]
we denote by $d(x)$ the cardinality of $B(1,x)$ and call it the degree of $x$. 

\begin{defin}
We say that a given $G$ is locally finite if $d(x)< \infty$ for all $x\in \mathcal{V}(G)$.
\end{defin}
From now on we assume $G$ to be locally finite.

Given real functions $f$ and $g$ on $\mathcal{V}(G)$ and $x\in\mathcal{V}(G)$ we define the following operators:
\begin{itemize}
  \item $\Delta (f)(x):=\sum_{v\in B(1,x)}(f(v)-f(x))$;
  \item $\Gamma (f,g)(x):=\frac{1}{2}\sum_{v\in B(1,x)}(f(x)-f(v))(g(x)-g(v))$;
  \item $\Gamma_2(f)(x):=\frac{1}{2}\left(\Delta(\Gamma(f,f))(x)\right)-\Gamma(f,\Delta(f))(x)$.
\end{itemize}
Instead of $\Gamma(f,f)(x)$ we write $\Gamma(f)(x)$. Note that $\Gamma(f)(x)\geq 0$ $\forall x\in \mathcal{V}(G)$, the equality holds if and only if $f(x)=f(v)$ for all $v$ in $B(1,x)$. Also note that $\Delta(f)(x)=\sum_{v\in B(1,x)}f(v)-d(x)f(x)$.

\medskip
These definitions allow us, following \cite{klart} and \cite{shmuck}, to introduce the Ricci curvature of a graph:
\begin{defin}\label{curvy}
The discrete Ricci curvature of a graph $G$, denoted $\ric(G)$, is the maximum value $K\in\mathbb{R}\cup \{-\infty\}$ such that for any real function $f$ on $\mathcal{V}(G)$ and any vertex $x$, $\Gamma_2(f)(x)\geq K\Gamma(f)(x)$ .
\end{defin}
There is also a local version of this definition: 
\begin{defin}
The local Ricci curvature of a graph $G$ at a given point $\bar{x}\in\mathcal{V}(G)$ is the maximum value $K\in\mathbb{R}\cup \{-\infty\}$ such that for any real function $f$ on $\mathcal{V}(G)$ , $\Gamma_2(f)(\bar{x})\geq K\Gamma(f)(\bar{x})$ holds. The local curvature so defined is denoted $\Ric(G)_{\bar{x}}$. 
\end{defin}

We obtain that $\Ric(G)=inf_{x\in\mathcal{V}(G)}\Ric(G)_x$.

For brevity, in the rest of this work, we often say "curvature" instead of "Discrete Ricci curvature".

\begin{obs}
Note that if $c,d\in\mathbb{R}$
and $f:\mathcal{V}(G)\rightarrow \mathbb{R}$
then $\Delta(f)=\Delta(f+c)$, $\Gamma(f+c,g+d)=\Gamma(f,g)$ and $\Gamma_{2}(f+c)=\Gamma_2(f)$. 
We may therefore assume that in the expression $\Gamma_2(f)(x)\geq K\Gamma(f)(x)$, $f$ satisfies $f(x)=0$. This allow us to use the following formula for $\Gamma$:
$$\Gamma (f)(x)=\frac{1}{2}\sum_{v\in B(1,x)}f(v)^2.$$
\end{obs}

The following result appears in \cite[Subsection 1.1]{klart}:

\begin{prop}\label{proof2}
$\Gamma_2$ can be expressed through the following formula when $f(x)=0$:
$$2\Gamma_2(f)(x)=\frac{1}{2}\sum_{u\in B(2,x)}\sum_{v\in B(1,u)\cap B(1,x)}(f(u)-2f(v))^2+\left( \sum_{v\in B(1,x)}f(v) \right)^2+
$$
$$
+\sum_{\{v,v'\}\in \mathcal{E}(G)}\left(2(f(v)-f(v'))^2+\frac{1}{2}(f(v)^2+f(v')^2)\right)+\sum_{v\in B(1,x)}\frac{4-d(x)-d(v)}{2} f(v)^2,
$$
where the third sum runs over all $v,v'\in B(1,x)$ such that $\{v,v'\}$ is an edge in $G$.
\end{prop}

The following simple observation does not appear anywhere in the literature so we include its proof:

\begin{lema}\label{proof3}
We can write
  \begin{equation}\label{equcurv}
  \text{Ric}(G)=\text{inf}_{x,f} \frac{\Gamma_2(f)(x)}{\Gamma(f)(x)} 
  \end{equation}
  where $x$ ranges in $\mathcal{V}(G)$ and $f$ ranges over the real functions defined on $\mathcal{V}(G)$ such that $f(x)=0$ and $\Gamma(f)(x)>0$.
\end{lema}

\begin{proof}
We know that $G$ has no isolated vertices, so there is a function $\bar{f}$ on $\mathcal{V}(G)$ such that $\Gamma(\bar{f})(x)> 0$. Any $K$ that satisfies 
$$\Gamma_2(f)(x)\geq K\Gamma(f)(x )\quad \forall x,\forall f$$  also satisfies
$$
K\leq \frac{\Gamma_2(g)(x)}{\Gamma(g)(x)}<\infty
\quad \forall x,g \text{ s.t.} \Gamma(g)(x)>0.
$$
Let now $\tilde{f}:\mathcal{V}(G)\rightarrow \mathbb{R}$ be such that $\Gamma(\tilde{f})(x)=0$, it comes easily from Proposition \ref{proof2} that in this case $\Gamma_2(\tilde{f})(x)\geq 0$ and so in particular $\Gamma_2(\tilde{f})(x)\geq K\Gamma(\tilde{f})(x)$.
The statement follows. 
\end{proof}

\begin{obs}
From Lemma \ref{proof3} we deduce the following formula for the local Ricci curvature:
\[
  \text{Ric}(G)_x:=\text{inf}_{f} \frac{\Gamma_2(f)(x)}{\Gamma(f)(x)},
\]
where $f$ ranges among the functions on $\mathcal{V}(G)$ such that $f(x)=0$ and $\Gamma(f)(x)\neq 0$.
\end{obs}

\begin{obs}\label{obs2}
From the formulas for $\Gamma$ and $\Gamma_2$ we obtain that $\Ric(G)_x$ depends only on the subgraph obtained as the union of the paths of length 1 and 2 starting from $x$. The set of vertices of this graph is $\{x\}\cup B(1,x)\cup B(2,x)$ and the edges are the ones connecting vertices in $B(1,x)$ to vertices in $\{x\}\cup B(1,x)\cup B(2,x)$. We will refer to such a subgraph as the length-2 path subgraph of $x$. As a consequence two vertices with isomorphic length-2 path subgraphs have the same local Ricci curvature.
\end{obs}

We conclude with the following remark about triangle-free graphs. These are graphs with no $x,u,v\in\mathcal{V}(G)$ such that $\{x,u\},\{x,v\},\{u,v\}\in \mathcal{E}(G)$:

\begin{obs}\label{triang}
If $G$ is a triangle-free graph, then:
\[
2\Gamma_2(f)(x)=\frac{1}{2}\sum_{u\in B(2,x)}\sum_{v\in B(1,u)\cap B(1,x)}(f(u)-2f(v))^2+
\]
\[
+(\sum_{v\in B(1,x)}f(v))^2+\sum_{v\in B(1,x)}\frac{4-d(x)-d(v)}{2} f(v)^2.
\]
We obtain this formula just by erasing the third sum in the equation in Proposition \ref{proof2}.
\end{obs}

We go on by recalling \cite[Theorem 1.2]{klart} and a corollary:

\begin{teo}\label{teotriang}
Let $G$ be a locally finite graph, $t(v,v')$ be the function that counts the number of triangles containing both the vertices $v$ and $v'$ and let $T=sup_{\{v,v'\}\subset \mathcal{V}(G)} t(v,v')$. Then $\Ric(G)\leq 2+\frac{T}{2}$.
\end{teo}

\begin{corollary}\label{coro}
Let $G$ be a graph with no triangles, then $\Ric(G)\leq2$.
\end{corollary}

We include now the statement of \cite[Theorems 27, 29]{sic3} which are crucial to prove main Theorem \ref{H(W)1}. 
\begin{teo}\label{teomatrix}
Given a locally finite graph $G$ and $x$ a vertex of $G$, then

\[
\Ric(G)_x=\min \{\lambda|\lambda \text{ is an eigenvalue of }A(x)\}.
\]
As a consequence
\[
\Ric(G)=\inf\{\lambda|\lambda \text{ is an eigenvalue of }A(x),x\in \mathcal{V}(G)\}.
\]
\end{teo}

\begin{teo}\label{corostima}
Let $G$ be a triangle free graph then 
\[
\Ric(G)\geq 4-\max_{x,y\in \mathcal{E}(G)} \left( \frac{3d(x)+d(y)}{2} \right)
\]

\end{teo}

%\subsubsection{Spectral gap and isoperimetry}

\subsection{Coxeter groups}\label{ssec1}

In this subsection we recall the definitions and main facts about Coxeter groups, the main reference is \cite{BB}. Our main goal is to define a family of graphs associated to Coxeter groups, namely the Hasse diagrams of the Bruhat order (denoted $H(W)$).

 Coxeter systems are pairs $(W,S)$ where $W$ is a group generated by the elements in $S$  and $S=\{s_i\}_{i\in I}$ is a finite set with the following relations:
\[
(s_is_j)^{m_{i,j}}=e.
\] 
The values $m_{i,j}$ are usually seen as the entries of a symmetric matrix with $m_{i,i}=1$ for all $i\in I$ and $m_{i,j}\geq 2$ (including $m_{i,j}=\infty$) for all $i,j\in I, i\neq j$. $W$ is called a Coxeter group, $S$ turns out to be a minimal set of generators, its elements are called Coxeter generators.
All the information about a Coxeter group can be encoded in a labeled graph called the Coxeter graph. Its set of vertices is $S$, and there is an edge between two vertices $s_i$ and $s_j$ if $m_{i,j}\geq 3$, such an edge is labeled with $m_{i,j}$ if $m_{i,j}\geq 4$. 
We state a very classical result from Coxeter Theory, namely the classification of finite Coxeter groups, for more details the reader can see \cite[Appendix A1]{BB} and \cite[Chapter 2]{hump}.
\begin{teo}
Given a finite Coxeter group $W$, this can be written in a unique way as direct product of the following irreducible Coxeter groups:
\[
W=W_1\times\ldots\times W_k
\] 
With $W_i$ of the following kind:
\begin{itemize}
\item $A_n$ $n\geq 1$;
\item $B_n$ $n\geq 2$;
\item $D_n$ $n\geq 3$;
\item $I_2(m)$ $m\geq 2$;
\item $H_3$, $H_4$;
\item $E_6$, $E_7$, $E_8$;
\item $F_4$.
\end{itemize}
\end{teo}
\begin{obs}
Groups of type $A_n$ are symmetric groups, in particular $A_{n}=S_{n+1}$. The groups of type $B_n$, called hyperoctahedral groups, are the groups of permutations $\pi$ of the set $\{\pm 1,\ldots,\pm n\}$ such that $\pi(a)=-\pi(a)$. Finally $D_n$ is the subgroup of $B_n$ of permutations such that an even number of positive elements has negative image. $D_n$ is called even hyperoctahedral group.
\end{obs}

We continue with some classical definitions in Coxeter theory. 
Given an element $w$ in $W$, this can be written as a product of elements in $S$
\[
w=s_1\ldots s_k. 
\]
If $k$ is the minimal length of all the possible expressions for $w$, we say that $k$ is the length of $w$ and we write $\ell(w)=k$. We define in $W$ the set of reflections as the union of all the conjugates of $S$, $T:=\cup_{w\in W}wSw^{-1}$. The definitions of length and reflections allow us to define a partial order on the set $W$:

\begin{defin}
Given $w\in W$ and $t\in T$, if $w'=tw$ and
$\ell(w')<\ell(w)$ we write $w'\rightarrow w$.
 Given two elements $v,w\in W$ we say that 
 $v \geq w$ according to the Bruhat order if there are 
 $w_0,\ldots, w_k\in W$ such that
\[
v=w_0\leftarrow w_1\ldots w_{k-1}\leftarrow w_k=w.
\]
\end{defin}
Thus we obtain an order on $W$. The undirected Bruhat graph associated to a Coxeter group, denoted $B(W)$, is the graph whose set of vertices is $W$ and such that there is an edge between two vertices $w,v$ if and only if $w\rightarrow v$ or $v\rightarrow w$.

Another graph associated to the Bruhat order is its Hasse graph. We denote it by $H(W)$, its vertices are the elements of $W$, there is an edge between two elements if and only if one covers the other according to the Bruhat order.
In the following proposition we describe the pairs of adjacent vertices in $H(A_n)$, in Section \ref{subsec2} we present the analogous results for $H(B_n)$ (Proposition \ref{coveringB}) and $H(D_n)$ (Proposition \ref{propdn}).

\begin{prop}\label{coveringA}
Let $\pi$ and $\sigma$ be two elements in $A_n$, then $\pi$ covers $\sigma$ if and only if there exist $1\leq i<k\leq n+1$ such that:
\begin{itemize}
\item$\sigma=(ab)\pi$ (namely $\pi=[\ldots,b,\ldots,a,\ldots]$ and $\sigma=[\ldots,a,\ldots,b,\ldots]$);
\item $b=\pi(i)>\pi(k)=a$;
\item there is no $i<j<k$ such that $a<\pi(j)<b$.
\end{itemize}
\end{prop}
A proof of Proposition \ref{coveringA} can be found in
\cite[Lemma 2.i.4]{BB}, yet with a slightly different, though equivalent,  statement. In Section \ref{subsec2} we also present results about the maximal number of edges of $H(B_n)$ and $H(D_n)$, an analogue result for type $A_n$ is proved in \cite{adinroich}.

\section{Ricci curvature of Hasse diagrams of the Bruhat order}
\label{subsec2}
This Section is devoted to the study of the Ricci curvature of the Hasse diagram associated to the Bruhat order of finite irreducible Coxeter groups. These graphs, that we denote by $H(W)$, where $W$ is a Coxeter group, whose set of vertices coincides with the set of elements of $W$. Two vertices are adjacent by an edge if one of the two corresponding elements of $W$ covers the other according to the strong Bruhat order.

Unlike the cases of Bruhat graphs and weak order graphs, in general the multiplication by an element of $W$ does not induce a graph automorphism and two elements may have non-isomorphic length-2 path subgraphs. Therefore the study of the global Ricci curvature doesn't coincide with the local study. For this reason most of the results of this Section are bounds of the global Ricci curvature. The only exception is the case of dihedral groups where we have been able to find exact values of the Ricci curvature of the Hasse diagrams. Even if it is not sufficient to work locally on a single vertex, in the case of dihedrals the structure of $H(I_2(m))$ allows us to recover the Ricci curvature studying only a finite number of length-2 path subgraphs.

\begin{prop}\label{H(W)dihedral}
The following identities hold for the Ricci curvature of the Hasse diagram of the dihedral groups:
\begin{itemize}
\item $\ric(H(I_2(3)))=\frac{21-\sqrt{33}}{12}$,
\item $\ric(H(I_2(4)))=\frac{1}{2}$,
\item $\ric(H(I_2(5)))=\frac{5-\sqrt{17}}{2}$,
\item $\ric(H(I_2(n)))=0$ for any $n>5$.
\end{itemize}
\end{prop}
\begin{proof}
We use Theorem \ref{teomatrix}. Notice that in $H(I_2(m))$  only six length-2 path subgraphs appear up to isomorphism, see Figures \ref{pica}-\ref{picf}.

\begin{figure}
\begin{minipage}[b]{14cm}
\begin{minipage}[b]{3.4 cm}
\centering
  \begin{tikzpicture}[scale=.6]
\foreach \Point in {(0,0), (1,1),(-1,1),(-1,2), (1,2)}{
    \node at \Point {\textbullet};
}
\node[shape=circle,draw=black] () at (0,0) {} ;
    \draw[thick] (0,0) -- +(-1,1);
    \draw[thick] (0,0) -- +(1,1);
%    \draw[thick] (0,0) -- +(0,3);
    \draw[thick] (1,1) -- +(-2,1);
    \draw[thick] (1,1) -- +(0,1);
    \draw[thick] (-1,1) -- +(2,1);
    \draw[thick] (-1,1) -- +(0,1);
%    \draw[thick] (-1,2) -- +(1,1);
%    \draw[thick] (1,2) -- +(-1,1);
  \end{tikzpicture}
\caption{A}\label{pica}
\end{minipage}
\ \hspace{2mm} \hspace{5mm} \
\begin{minipage}[b]{3.4cm}
\centering
  \begin{tikzpicture}[scale=.6]
\foreach \Point in {(0,0), (1,1),(-1,1),(-1,2), (1,2),(0,3)}{
    \node at \Point {\textbullet};
}
\node[shape=circle,draw=black] () at (-1,1) {} ;
    \draw[thick] (0,0) -- +(-1,1);
    \draw[thick] (0,0) -- +(1,1);
%    \draw[thick] (0,0) -- +(0,3);
    \draw[thick] (1,1) -- +(-2,1);
    \draw[thick] (1,1) -- +(0,1);
    \draw[thick] (-1,1) -- +(2,1);
    \draw[thick] (-1,1) -- +(0,1);
    \draw[thick] (-1,2) -- +(1,1);
    \draw[thick] (1,2) -- +(-1,1);
  \end{tikzpicture}
\caption{B}
\end{minipage}
\ \hspace{5mm} \hspace{5mm} \
\begin{minipage}[b]{3.4cm}
\centering
  \begin{tikzpicture}[scale=.6]
\foreach \Point in {(0,0), (1,1),(-1,1),(-1,2), (1,2),(-1,3),(1,3)}{
    \node at \Point {\textbullet};
}
\node[shape=circle,draw=black] () at (-1,1) {} ;
    \draw[thick] (0,0) -- +(-1,1);
    \draw[thick] (0,0) -- +(1,1);
%    \draw[thick] (0,0) -- +(0,3);
    \draw[thick] (1,1) -- +(-2,1);
    \draw[thick] (1,1) -- +(0,2);
    \draw[thick] (-1,1) -- +(2,1);
    \draw[thick] (-1,1) -- +(0,2);
    \draw[thick] (-1,2) -- +(2,1);
    \draw[thick] (1,2) -- +(-2,1);
  \end{tikzpicture}
\caption{C}
\end{minipage}
\end{minipage}
\end{figure}

\begin{figure}
\begin{minipage}[b]{14cm}
\begin{minipage}[b]{3.4 cm}
\centering
  \begin{tikzpicture}[scale=.6]
\foreach \Point in {(0,0), (1,1),(-1,1),(-1,2), (1,2),(-1,3),(1,3),(0,4)}{
    \node at \Point {\textbullet};
}
\node[shape=circle,draw=black] () at (-1,2) {} ;
    \draw[thick] (0,0) -- +(-1,1);
    \draw[thick] (0,0) -- +(1,1);
%    \draw[thick] (0,0) -- +(0,3);
    \draw[thick] (1,1) -- +(-2,1);
    \draw[thick] (1,1) -- +(0,2);
    \draw[thick] (-1,1) -- +(2,1);
    \draw[thick] (-1,1) -- +(0,2);
    \draw[thick] (-1,2) -- +(2,1);
    \draw[thick] (1,2) -- +(-2,1);
    \draw[thick] (-1,3) -- +(1,1);
    \draw[thick] (1,3) -- +(-1,1);
  \end{tikzpicture}
\caption{D}
\end{minipage}
\ \hspace{2mm} \hspace{5mm} \
\begin{minipage}[b]{3.4 cm}
\centering
  \begin{tikzpicture}[scale=.6]
\foreach \Point in {(0,0), (1,1),(-1,1),(-1,2), (1,2),(-1,3),(1,3),(1,4),(-1,4)}{
    \node at \Point {\textbullet};
}
\node[shape=circle,draw=black] () at (-1,2) {} ;
    \draw[thick] (0,0) -- +(-1,1);
    \draw[thick] (0,0) -- +(1,1);
%    \draw[thick] (0,0) -- +(0,3);
    \draw[thick] (1,1) -- +(-2,1);
    \draw[thick] (1,1) -- +(0,3);
    \draw[thick] (-1,1) -- +(2,1);
    \draw[thick] (-1,1) -- +(0,3);
    \draw[thick] (-1,2) -- +(2,1);
    \draw[thick] (1,2) -- +(-2,1);
    \draw[thick] (-1,3) -- +(2,1);
    \draw[thick] (1,3) -- +(-2,1);
  \end{tikzpicture}
\caption{E}
\end{minipage}
\ \hspace{5mm} \hspace{5mm} \
\begin{minipage}[b]{3.4 cm}
\centering
  \begin{tikzpicture}[scale=.6]
\foreach \Point in {(1,0),(-1,0),(1,1),(-1,1),(-1,2), (1,2),(-1,3),(1,3),(1,4),(-1,4)}{
    \node at \Point {\textbullet};
}
\node[shape=circle,draw=black] () at (-1,2) {} ;
%    \draw[thick] (0,0) -- +(-1,1);
    \draw[thick] (1,0) -- +(-2,1);
    \draw[thick] (-1,0) -- +(2,1);
    \draw[thick] (1,1) -- +(-2,1);
    \draw[thick] (1,0) -- +(0,4);
    \draw[thick] (-1,1) -- +(2,1);
    \draw[thick] (-1,0) -- +(0,4);
    \draw[thick] (-1,2) -- +(2,1);
    \draw[thick] (1,2) -- +(-2,1);
    \draw[thick] (-1,3) -- +(2,1);
    \draw[thick] (1,3) -- +(-2,1);
  \end{tikzpicture}
\caption{F}\label{picf}
\end{minipage}
\end{minipage}
\end{figure}

The matrices associated to these neighbourhoods according to Theorem \ref{teomatrix} are listed below, together with their spectrum:\\
 \begin{tabular}{cc}
\begin{minipage}{0.4\textwidth}

\[
M_A=
\begin{bmatrix}
\frac{5}{2} & -1         \\
-1          &\frac{5}{2}
\end{bmatrix}
\]
\[
\Spec(M_A)=\biggl[\frac{7}{2},\frac{3}{2} \biggr].
\]

%\caption{Tetrahedron}
\end{minipage}
%\pause
&
\begin{minipage}{0.4\textwidth}
\[
M_B=
\begin{bmatrix}
\frac{7}{3} & -\frac{2}{3} & \frac{1}{3}\\
-\frac{2}{3}& \frac{7}{3}  &\frac{1}{3} \\
\frac{1}{3} & \frac{1}{3}  &\frac{11}{6}
\end{bmatrix}
\]
\[
\Spec(M_B)=\biggl[3, \frac{21-\sqrt{33}}{12},\frac{21+\sqrt{33}}{12} \biggr]
\]
\end{minipage}
\end{tabular}

 \begin{tabular}{cc}
\begin{minipage}{0.4\textwidth}
\[
M_C=
\begin{bmatrix}
\frac{17}{6} & -\frac{5}{3} & \frac{1}{3}\\
-\frac{5}{3}& \frac{17}{6}  & \frac{1}{3}\\
\frac{1}{3} & \frac{1}{3}   & \frac{11}{6} 
\end{bmatrix}.
\]
\[
\Spec(M_C)=\biggl[\frac{9}{2},\frac{3}{2}+\frac{\sqrt{3}}{3},\frac{3}{2}-\frac{\sqrt{3}}{3}\biggr]
\]
\end{minipage}
&
\begin{minipage}{0.4\textwidth}
\[
M_D=
\begin{bmatrix}
2 & -\frac{1}{2}&\frac{1}{2} & \frac{1}{2} \\    
-\frac{1}{2} & 2 & \frac{1}{2} &\frac{1}{2}\\
\frac{1}{2} & \frac{1}{2} & 2 & -\frac{1}{2}\\
\frac{1}{2} & \frac{1}{2} & -\frac{1}{2}& 2             \end{bmatrix}
\]
\[
\Spec(M_D)=\biggl[\frac{5}{2},\frac{5}{2},\frac{5}{2},\frac{1}{2}\biggr].
\]
%\caption{Tetrahedron}
\end{minipage}
\end{tabular}

 \begin{tabular}{cc}
\begin{minipage}{0.4\textwidth}
\[
M_E=
\begin{bmatrix}
\frac{5}{2} & -\frac{3}{2} & \frac{1}{2} &\frac{1}{2}\\
-\frac{3}{2}& \frac{5}{2}  &\frac{1}{2} & \frac{1}{2} \\
\frac{1}{2} & \frac{1}{2}  & 2 & -\frac{1}{2}\\
\frac{1}{2} & \frac{1}{2} & -\frac{1}{2} & 2
\end{bmatrix}
\]
\[
\Spec(M_E)\biggl[4,\frac{5}{2},\frac{5+\sqrt{17}}{4},\frac{5-\sqrt{17}}{4}\biggr]
\]
\end{minipage}
&
%\pause
\begin{minipage}{0.4\textwidth}
\[
M_F=
\begin{bmatrix}
\frac{5}{2} & -\frac{3}{2} & \frac{1}{2}& \frac{1}{2}\\
-\frac{3}{2}& \frac{5}{2} & \frac{1}{3} & \frac{1}{2}\\
\frac{1}{2} & \frac{1}{2} & \frac{5}{2} & -\frac{3}{2}\\
\frac{1}{2} & \frac{1}{2} & -\frac{3}{2} &
\frac{5}{2}
\end{bmatrix}.
\]
\[
\Spec(M_F)=\biggl[4,4,2,0\biggr]
\]
\end{minipage}
\end{tabular}

And the result follows.
\end{proof}

Our next goal is to work with the discrete curvature of the Hasse diagrams of types $A_n$, $B_n$ and $D_n$.
The idea we used in Proposition \ref{H(W)dihedral}
is of no help here, it is actually quite simple to show that in these cases there are infinitely many isomorphism classes of length-2 path subgraphs.
The result we obtained for type $A,B$ and $D$
are the inequalities stated below:

\begin{teo}\label{H(W)1}
The following inequalities hold for the discrete Ricci curvature  of the graphs $H(A_n), H(B_n)$ and $H(D_n)$.
$$
\Ric(H(A_{n-1}))\geq -\lfloor \frac{n^2}{2}\rfloor-2n+8;
$$ 
$$
\Ric(H(B_n))\geq 4(-2n+1)   \text{  for }n\leq 5,\quad
\Ric(H(B_n)) \geq -2\lfloor\frac{n^2}{2}\rfloor-2n+6
   \text{  for }n\geq 5;
$$ 
$$
\Ric(H(D_n)) \geq -2\lfloor\frac{n^2}{2}\rfloor-2n+6.
$$
\end{teo}

\begin{proof}
The graphs are all triangle free, this allows us to use Theorem \ref{corostima}, we can bound $\Ric(G)$ as follows
\begin{equation}\label{sopra}
 \Ric(G)\geq 4-\frac{3d(x)+d(y)}{2}\geq 4-2d_{\max}(G),
\end{equation}
 where $d_{\max}(G)$ is the maximum degree of a vertex in $G$. For $A_{n-1}$ it is known from \cite[section 3]{adinroich} that $d_{max}(H(A_{n-1}))=\lfloor \frac{n^2}{4}\rfloor+n-2$. For type $B_n$ the maximum degree of a vertex is $4(n-1)$ for $n\leq 5$ and 
$\lfloor \frac{n^2}{2}\rfloor+n-1$ otherwise; this is proved in Section \ref{dmaxbn}. For type $D_n$ the maximum degree of an element in the Hasse graph is $\lfloor \frac{n^2}{2}\rfloor+n-1$, this is proved in Section \ref{dmaxdn}. Applying these results in equation (\ref{sopra}) we conclude.
\end{proof}

A similar reasoning leads to a bound of the Ricci curvature of Hasse diagrams of some exceptional irreducible Coxeter group.

\begin{prop}\label{H(W)}
The following inequalities hold for the discrete Ricci curvature  of the graphs $H(F_4),H(E_6)$ and $H(H_3)$.
$$
\Ric(H(F_4))\geq -28;
$$ 
$$
\Ric(H(E_6))\geq -46;
$$ 
$$
\Ric(H(H_3)) \geq -14.
$$
\end{prop}
\begin{proof}
We apply again inequality \ref{sopra} used to prove the above theorem. This time the values for the maximal degree of the graphs are computed using bruhat\textunderscore upper\textunderscore
covers and bruhat\textunderscore lower \textunderscore covers on SageMath (\cite{sage}) (for $E_6$ and $F_4$) and program Mongelli2  for Maple \cite{maple} (for $H_3$). We obtain the following values:
\[
d_{max}(H(F_4))=16;\quad d_{max}(H(E_6))=25;\quad d_{max}(H(H_3))=9.
\]
An the statement follows.
\end{proof}

\subsection{On the maximum degree in the Hasse diagram of hyperoctahedral group}
\label{dmaxbn}

In this section we study the maximal degree in the Hasse diagram of Bruhat order in $B_n$ ($H(B_n)$). The result that we obtain is the following:
\begin{teo}\label{teobn}
The maximum degree of the Hasse diagram of the strong Bruhat order in $B_n$ with $n\geq 5$ is 
\[
\lfloor \frac{n^2}{2}\rfloor+n-1.
\]
\end{teo}

First we start by studying the covering relations in such a graph:

\begin{prop}\label{coveringB}
Given $\pi$ and $\sigma$ in $B_n$, we have that $\sigma$ covers $\pi$ according to the Bruhat order if and only if $\sigma=s \pi$ where $s$ is a reflection of the following three types:
\begin{itemize}
\item[1)] $s=(\pi(i),\pi(j))(-\pi(i),-\pi(j))$ with $0<i<j$, $\pi(i)<\pi(j)$ and for every $i<k<j$ $	\pi(k)\notin [\pi(i),\pi(j)]$; 
\item[2)]  $s=(\pi(i),-\pi(i))$ such that $i>0$, $\pi(i)>0$ and for every $j$, $0<j<i$ $\pi(j)\notin [-\pi(i),\pi(i)]$;
\item[3)] $s=(\pi(i),-\pi(j))(\pi(j),-\pi(i))$ with $i,j>0$, $sign(\pi(i))=sign(-\pi(j))$, $-\pi(j)<\pi(i)$ and such that  given $k\in [-j,i]$, $\pi(k)>\pi(i)$ or $\pi(k)<-\pi(j)$.
\end{itemize}
\end{prop}

To prove this Proposition first we need the following fact which is a consequence of \cite[Corollary 8.1.9]{BB}:
\begin{prop}\label{s2n}
The Bruhat order on $B_n$ is a subposet of $S_{2n}=S_{[-n,\ldots,n]}$ and
\[
u\leq v \text{ in }B_n \Leftrightarrow u\leq v \text{ in }S_{[-n,\ldots,n]}
\]
for any $u,v\in B_n\subset S_{[-n,\ldots,n]}$.
\end{prop}
Therefore we recall the result of Proposition \ref{coveringA}:
\begin{prop}
Let $\pi$ and $\sigma$ be two elements in $A_n$, $\pi$ covers $\sigma$ if and only if there exist $1\leq i<k\leq n+1$:
\begin{itemize}
\item$\sigma=(ab)\pi$, $\pi=[\ldots,b,\ldots,a,\ldots]$ and $\sigma=[\ldots,a,\ldots,b,\ldots]$.
\item $b:=\pi(i)>\pi(k)=a$
\item There is no $i<j<k$ such that $a<\pi(j)<b$.
\end{itemize}
\end{prop}

Now we can give the following proof:
\begin{proof}[Proof of Theorem \ref{coveringB}]
 $\sigma$ covers $\pi$ in $B_n$ if and only if $\sigma> \pi$ in the strong Bruhat order of $S_{2n}=S_{[-n,\ldots,n]}$ and there are no elements of $\theta\in B_{n}\subset S_{2n}=S_{[-n,\ldots,n]}$ such that $\sigma>\theta>\pi$.

We consider $s$ of type $1$, we notice that in $S_{2n}$ $\sigma>\pi$
and $\ell(\sigma)-\ell(\pi)=2$. This means that the interval $[\pi,\sigma]$ has the following shape (\cite[Lemma 2.7.3]{BB}):

\begin{center}
\begin{tikzpicture}[scale=1]
     \draw (0,0) node[anchor=north]  {$\pi$};
     \draw (-0.5,1) node[anchor=east]  {$(a,b)\pi$};
     \draw (0.5,1) node[anchor=west]  {$(-a,-b)\pi$};
     \draw (0,2) node[anchor=east]  {$\sigma$};
\foreach \Point in {(0,0), (0,2),(-0.5,1),(0.5,1)}{
    \node at \Point {\textbullet};
}
    \draw[thick] (0,0) -- +(-0.5,1);
    \draw[thick] (0,0) -- +(0.5,1);
    \draw[thick] (0,2) -- +(0.5,-1);
    \draw[thick] (0,2) -- +(-0.5,-1);
  \end{tikzpicture}
\end{center} 
It follows that there are no elements of $B_n$ strictly contained in $[\pi,\sigma]$.  

Assume now $s$ of type $2$. In this case $\sigma$ covers $\pi$ in $B_n$ if and only if it covers $\pi$ in $S_{2n}$, and this happens if and only if the condition required for $0<j<i$ is verified.

We consider at last the case of $s$ a type $3$ reflection. Reasoning as for type $1$ reflections we obtain that $\sigma$ covers $\pi$. Also, assuming $sign(\pi(i))\neq sign(\pi(j))$ the condition on $k$ is necessary and sufficient in order to have a covering.

Now it is left to prove that in case of $(\pi(i),-\pi(j))(-\pi(i),\pi(j))$ $i,j>0$, if $sign(\pi(i))= sign(\pi(j))$ $\sigma$ doesn't cover $\pi$. Clearly $\sigma$ doesn't cover $\pi$ if there is a $-j<k<i$ such that $min(\pi(i),-\pi(j))<\pi(k)<max(\pi(i),-\pi(j))$. We assume that such a $k$ doesn't exist. Now it is sufficient to notice that $s'=(\pi(i),-\pi(i))$ satisfies the hypothesis of case $2$. 
In $S_{2n}$ we have that $\pi< s'\pi< \sigma$, in particular $\pi$ is not covered by $\sigma$.
\end{proof}

Given an element $\pi\in B_n$, we associate to it two  graphs. The first one is $\Gamma(\pi)$ with set of vertices $\{1,\ldots,n\}$ and edges as follows:
\begin{itemize}
\item an undirected edge between $a$ and $b$ if $|\ell((a,b)(-a,-b)\pi)-\ell(\pi)|=1$, which is to say that $\{(a,b)(-a,-b)\pi,\pi\}$ is an edge in the Hasse diagram of $B_n$;
\item a loop at $a$ if $(a,-a)$ is an edge starting from $\pi$ in the Bruhat order of $B_{n}$;
\item an undirected edge between $a$ and $b$ labeled with a minus sign ($-$) if $|\ell((a,-b)$ $(-a,b)\pi)$ $-\ell(\pi)|=1$ .
\end{itemize}

The second graph is $\tilde{\Gamma}(\pi)$ and is defined as follows:
\begin{itemize}
\item the set of vertices is $\{\pm 1,\ldots,\pm n\}$;
\item there is an undirected edge from $i$ to $j$ if $j\neq \pm i$ and $(i,j)(-i,-j)\pi$ is adjacent to  $\pi$ in $H(B_n)$;
\item there is an undirected edge from $i$ to $-i$ if  $(i,-i)$ is adjacent to $\pi$ in $H(B_n)$.
\end{itemize}

\begin{obs}\label{archigamma}
We note that given a vertex $i$ of $\Gamma(\pi)$ then 
\[
d_{\Gamma}(i)=d_{\tilde{\Gamma}}(i)=d_{\tilde{\Gamma}}(-i).
\]

Also, the total number of edges in $\Gamma(\pi)$ coincides with the degree of $\pi$ in $H(B_n)$.
\end{obs}

\begin{example}
We give an example of $\Gamma(\pi)$ and $\tilde{\Gamma}(\pi)$ when $\pi=[4,-3,2,-1]\in B_{4}$. The coverings of $\pi$ are $(4,-4)\pi$, $(2,-2)\pi$, $(2,3)(-2,-3)\pi$, $(1,3)(-1,-3)\pi$, $(1,-2)(-1,2)\pi$ and $(3,-4)(-3,4)\pi$. The elements covered by $\pi$ are 
$(3,-3)\pi$, $(1,-1)\pi$, $(3,4)(-3,-4)\pi$ $(2,4)(-2,-4)\pi$, $(1,2)(-1,-2)\pi$ and $(-2,3)(2,-3)\pi$. Therefore the graphs associated to $\pi$ are the ones in Figures \ref{gamma}, \ref{tildegamma} (the edges labeled with a minus sign are represented with a dotted line).

\begin{figure}
\begin{minipage}[b]{7cm}
\centering
\begin{tikzpicture}[scale=1]
     \draw (0,-1) node[anchor=north]  {$3$};
     \draw (0,1) node[anchor=west]  {$1$};
     \draw (-1,0) node[anchor=north]  {$4$};
     \draw (1,0) node[anchor=north]  {$2$};
\foreach \Point in {(0,-1), (0,1),(-1,0),(1,0)}{
    \node at \Point {\textbullet};
}
\path
(1,0) edge [bend right] node {} (0,-1)
(1,0) edge [dotted, bend left] node {} (0,-1)
(0,1) edge [bend right] node {} (1,0)
(0,1) edge [dotted, bend left] node {} (1,0)
(0,-1) edge [bend right] node {} (-1,0)
(0,-1) edge [dotted, bend left] node {} (-1,0)
(1,0)  edge [loop] node {} (1,0)
(-1,0)  edge [loop] node {} (-1,0)
(0,1)  edge [loop] node {} (0,1)
(0,-1)  edge [loop] node {} (0,-1)
(0,1) edge [] node {-} (0,-1)
(1,0) edge [] node {-} (-1,0);
\end{tikzpicture}
\caption{$\Gamma([4,-3,2,-1])$}\label{gamma}
\end{minipage}
\ \hspace{2mm} \hspace{2mm} \
\begin{minipage}[b]{5cm}
\begin{tikzpicture}[scale=1]
     \draw (-1.39,0.57) node[anchor=east]  {$4$};
     \draw (-1.39,-0.57) node[anchor=east]  {$-4$};
     \draw (1.39,0.57) node[anchor=west]  {$2$};
     \draw (1.39,-0.57) node[anchor=west]  {$-2$};
     \draw (0.57,-1.39) node[anchor=north]  {$3$};
     \draw (-0.57,-1.39) node[anchor=north]  {$-3$};
     \draw (0.57,1.39) node[anchor=south]  {$-1$};
     \draw (-0.57,1.39) node[anchor=south]  {$1$};
\foreach \Point in {(0.57,1.39),(-0.57,1.39),(-0.57,-1.39),(0.57,-1.39),(1.39,-0.57),(1.39,0.57),(-1.39,-0.57),(-1.39,0.57)}{
    \node at \Point {\textbullet};
}
\path
(0.57,1.39) edge [] node {} (1.39,-0.57)
(0.57,1.39) edge [] node {} (1.39,0.57)
(0.57,1.39) edge [] node {} (-0.57,-1.39)
(-0.57,1.39) edge [] node {} (1.39,-0.57)
(-0.57,1.39) edge [] node {} (1.39,0.57)
(-0.57,1.39) edge [] node {} (0.57,-1.39)
(1.39,0.57) edge [] node {} (0.57,-1.39)
(1.39,0.57) edge [] node {} (-0.57,-1.39)
(1.39,0.57) edge [] node {} (-1.39,0.57)
(1.39,-0.57) edge [] node {} (0.57,-1.39)
(1.39,-0.57) edge [] node {} (-0.57,-1.39)
(1.39,-0.57) edge [] node {} (-1.39,-0.57)
(0.57,-1.39) edge [] node {} (-1.39,0.57)
(0.57,-1.39) edge [] node {} (-1.39,-0.57)
(-0.57,-1.39) edge [] node {} (-1.39,0.57)
(-0.57,-1.39) edge [] node {} (-1.39,-0.57)
(0.57,1.39) edge [bend right] node {} (-0.57,1.39)
(-0.57,-1.39) edge [bend right] node {} (0.57,-1.39)
(-1.39,-0.57) edge [bend left] node {} (-1.39,0.57)
(1.39,0.57) edge [bend left] node {} (1.39,-0.57);
\end{tikzpicture}
\caption{$\tilde{\Gamma}([4,-3,2,-1])$}\label{tildegamma}
\end{minipage}
\end{figure}
\end{example}

We begin with the following result, which is the analogue, in type $B$, of \cite[Lemma 3.2]{adinroich}.

\begin{lema}\label{lemman+1}
Let $\pi\in B_n$, there is a vertex in $\Gamma(\pi)$ which has degree at most $n+1$.
\end{lema}
\begin{proof}
By Remark \ref{archigamma}, it is sufficient to prove the corresponding statement for $\tilde{\Gamma}(\pi)$.

Assume without loss of generality that $\pi(n)$ is a positive number that labels a vertex in $\tilde{\Gamma}(\pi)$ .
First we notice that the vertices adjacent to $\pi(n)$ can be divided in three sets: $U$, $L$ and $C$. $U$ is the set of $\pi(i)\neq \pm \pi(n)$ such that $\pi$ covers $(\pi(i),\pi(n))(-\pi(i),-\pi(n))$ $\pi$. This implies that given $\pi(i),\pi(i')\in U$,  then $\pi(i),\pi(i')>\pi(n)$ and $\pi(i)>\pi(i')$ if $i>i'$. 
$L$ is the set of $\pi(i)\neq \pm n$ such that $(\pi(i),\pi(n))(-\pi(i),-\pi(n))\pi$ covers $\pi$. This gives a decreasing sequence of $\pi(i)$, notice that $\pi(i)<\pi(n)$. In particular $\pi(i)$ may be negative.
$C$ is the set containing $-\pi(n)$ in case $(\pi(n),-\pi(n))\pi$ covers $\pi$, otherwise it is the empty set. The sets $U$, $L$ and $C$ form a partition of the edges starting from $\pi(n)$.

We assume by contradiction that $d(\pi(n))>n+1$, and fix $p$ to be the smallest element such that $\pi(p)\in U\cup L \cup C$. Notice that if $C$ is non-empty $p=-n$, otherwise $\pi(p)\in U\cup L$ and we expect it to be negative. 
 We claim that $\pi(p)$  has degree smaller than $n+1$.
$\pi(p)$ is adjacent to at most one element in $U$ and at most one in $L$ therefore the degree of $\pi(p)$ in $\tilde{\Gamma}(n)$ is:
\[
d(\pi(\pi))\leq  2n-\underbrace{d(\pi(n))}_{|B(1,\pi(n))|}+\underbrace{2}_{|B({1,\pi(n)})\cap B(1,\pi(p))|}
\leq 2n-n-2 +3=n.
\] 
This proves the statement.
\end{proof}

If $n$ is odd we can say more.

\begin{figure}
  \centering
   \begin{tikzpicture}[scale=.8]
      \draw (4.75,-2.25) node[anchor=west]  {\footnotesize $E$};
      \draw (3.5,-0.85) node[anchor=west]  {\footnotesize $A$};
      \draw (1,0.3) node[anchor=west]  {\footnotesize $F$};
      \draw (4.5,4) node[anchor=west]  {\footnotesize $C$}; 
      \draw (-3.25,2.5) node[anchor=east]  {\footnotesize $B$}; 
      \draw (-3.5,.8) node[anchor=east]  {\footnotesize $D$}; 
      \draw (-5,-2) node[anchor=west]  {$-\pi(n)$};  %   
      \draw (5,2) node[anchor=west]  {$\pi(n)$};  %
      \draw (-3.7,1.5) node[anchor=south]  {\footnotesize $\pi(p)$};
      \draw (3.3,-1.5) node[anchor=north]  {\footnotesize $-\pi(p)$};
\foreach \Point in {(5,2),%\pi(n)
                       (-5,-2),%\pi(n)
                              (3.5,-1.5), %-\pi(p)
                              (-3.5,1.5), %\pi(p)
               %elementi in D
               (-2.7,1.1),(-2.3,0.9),(-1,0.6), 
              % elementi in F 
               (0.4,0.4),(0.8,0.2),
               %elementi in A
               (2.7,-1.1),(2.3,-0.9),(1.3,-0.6),
               %elementi in E
               (4.1,-1.9),(4.4,-2.3),(4.7,-2.7),               
               %elementi in B
               (-.4,2.8),(-1.2,2.6),(-1.8,2.2),(-2.7,2.1),
               % elementi in C
               (0.6,3.4),(1.1,3.6),(2.3,4),(2.9,4.3),(4.1,4.5)
               %                               (-1,0.75), %più basso elemento in b
%                              (-3.5,2.25), %più alto elemento in b
%(1,0),(2,0),(1,1),(2,1),(2,2)
}{
     \node at \Point {\textbullet};
}
     \draw[thick] (-4.5,0) -- +(10,0); %yes
     \draw[thick] (0,-4) -- +(0,9); %yes
     \draw[dashed] (-4.5,1.5) -- +(10,0); %yes
     \draw[dashed] (-4.5,-1.5) -- +(10,0); %yes
% rettangolo che delimita D
     \draw[dotted] (-3.5,1.5) -- +(3.25,0);
     \draw[dotted] (-.25,.25) -- +(-3.25,0);
     \draw[dotted] (-.25,.25) -- +(0,1.25);
     \draw[dotted] (-3.5,1.5) -- +(0,-1.25);
% rettangolo che delimita B
     \draw[dotted] (-3.25,3) -- +(3,0);
     \draw[dotted] (-.25,2) -- +(-3,0);
     \draw[dotted] (-.25,2) -- +(0,1);
     \draw[dotted] (-3.25,3) -- +(0,-1);
% rettangolo che delimita A 
     \draw[dotted] (3.5,-1.5) -- +(-2.25,0);
     \draw[dotted] (1.25,-.25) -- +(2.25,0);
     \draw[dotted] (1.25,-.25) -- +(0,-1.25);
     \draw[dotted] (3.5,-1.5) -- +(0,1.25);
% rettangolo che delimita E 
     \draw[dotted] (3.75,-1.75) -- +(1,0);
     \draw[dotted] (4.75,-2.75) -- +(-1,0);
     \draw[dotted] (4.75,-2.75) -- +(0,1);
     \draw[dotted] (3.75,-1.75) -- +(0,-1);
% rettangolo che delimita F 
     \draw[dotted] (0.1,.5) -- +(.9,0);
     \draw[dotted] (1,0.1) -- +(-.9,0);
     \draw[dotted] (1,0.1) -- +(0,.4);
     \draw[dotted] (0.1,.5) -- +(0,-.4);
% rettangolo che delimita C 
     \draw[dotted] (0.2,3.2) -- +(4.3,0);
     \draw[dotted] (4.5,4.8) -- +(-4.3,0);
     \draw[dotted] (4.5,4.8) -- +(0,-1.6);
     \draw[dotted] (0.2,3.2) -- +(0,1.6);
   \end{tikzpicture}
\caption{Part 1}
\label{picturebn1}
\end{figure}

\begin{figure}
  \centering
   \begin{tikzpicture}[scale=.8]
      \draw (-3.8,1.9) node[anchor=east]  {$-A$};
      \draw (3.8,-1.9) node[anchor=west]  {$A$};
      \draw (-4.3,-2.5) node[anchor=east]  {$-B$};
      \draw (4.3,2.5) node[anchor=west]  {$B$}; 
      \draw (3.9,3.5) node[anchor=north]  {\footnotesize $\pi(b)$}; 
      \draw (-3.9,-3.5) node[anchor=south]  {\footnotesize $-\pi(b)$};
      \draw (-1.2,4.1) node[anchor=south]  {\footnotesize $\pi(p)$}; 
      \draw (1.2,-4.1) node[anchor=north]  {\footnotesize $-\pi(p)$};       
      \draw (-5,-.7) node[anchor=north]  {$-\pi(n)$};  %   
      \draw (5,.7) node[anchor=south]  {$\pi(n)$};  %
\foreach \Point in {(5,.7),%\pi(n)
                    (-5,-.7),%\pi(n)
               %elementi in A
               (3.4,-2.3),(2.3,-1.9),(1.6,-1.6),
               %elementi in -A
               (-3.4,2.3),(-2.3,1.9),(-1.6,1.6),
               %elementi in -B
               (.4,-1.8),(1.2,-1.2),
               % elementi in -C
               (-0.6,-2.4),(-1.1,-2.6),(-2.3,-3),             			   (-2.9,-3.3),(-4.1,-3.5),
               %elementi in B
               (-.4,1.8),(-1.2,1.2),
               % elementi in C
               (0.6,2.4),(1.1,2.6),(2.3,3),			               (2.9,3.3),(4.1,3.5)
}{
     \node at \Point {\textbullet};
}
     \draw[thick] (-4.5,0) -- +(10,0); %yes
     \draw[thick] (0,-4) -- +(0,9); %yes
     \draw[dashed] (-5,.7) -- +(10,0); %yes
     \draw[dashed] (-5,-.7) -- +(10,0); %yes
% rettangolo che delimita A 
     \draw[dotted] (0.8,-1) -- +(3,0);
     \draw[dotted] (3.8,-2.8) -- +(-3,0);
     \draw[dotted] (3.8,-2.8) -- +(0,1.8);
     \draw[dotted] (.8,-1) -- +(0,-1.8);
% rettangolo che delimita -A 
     \draw[dotted] (-0.8,1) -- +(-3,0);
     \draw[dotted] (-3.8,2.8) -- +(3,0);
     \draw[dotted] (-3.8,2.8) -- +(0,-1.8);
     \draw[dotted] (-.8,1) -- +(0,1.8);
% rettangolo che delimita B 
     \draw[dotted] (-1.4,1.1) -- +(5.7,0);
     \draw[dotted] (4.3,3.9) -- +(-5.7,0);
     \draw[dotted] (4.3,3.9) -- +(0,-2.8);
     \draw[dotted] (-1.4,1.1) -- +(0,2.8);
% rettangolo che delimita B 
     \draw[dotted] (1.4,-1.1) -- +(-5.7,0);
     \draw[dotted] (-4.3,-3.9) -- +(5.7,0);
     \draw[dotted] (-4.3,-3.9) -- +(0,2.8);
     \draw[dotted] (1.4,-1.1) -- +(0,-2.8);
%freccine
     \draw [draw,->] (-1.2,4.1) -- +(0,-2.6);
     \draw [draw,->] (1.2,-4.1) -- +(0,2.6);
   \end{tikzpicture}
\caption{Part 2}
\label{picturebn2}
\end{figure}

\begin{teo}
For any $\pi \in B_n$ with $n\geq 7$ and odd there is a vertex in $\Gamma(\pi)$ with degree at most $n$.
\end{teo}
\begin{proof}
As for the previous lemma, we prove the statement for $\tilde{\Gamma}(\pi)$.
We consider the vertex $\pi(n)\in \tilde{\Gamma}(\pi)$ that without loss of generality can be assumed positive; we suppose that $d(\pi(i))>n$ for any $i\in \{\pm1,\ldots,\pm n\}$.
First we notice a general fact that is useful for the rest of the proof. Assume that $\pi(i)$ and $\pi(n)$ are two vertices in $\tilde{\Gamma}(\pi)$ both with degree bigger than $n$, then:
\begin{align*}
|\{\text{elements not adjacent to }\pi(i) \text{or} \pi(n)\}|&= 2n-|B(1,\pi(i))|-|B(1,\pi(n))|+\\
&+|B(1,\pi(n))\cap B(1,\pi(i))|\\
&\leq 2n-2(n+1)+\\&+|B(1,\pi(n))\cap B(1,\pi(i))|\\
&=|B(1,\pi(n))\cap B(1,\pi(i))|-2.
\end{align*}
This means that the set $B(1,\pi(n))\cap B(1,\pi(i))$ has cardinality at least $2$.

The rest of the proof is divided in two parts: in the first one (see Figure \ref{picturebn1}) we assume that $\pi(n)$ is not adjacent to $-\pi(n)$ in $\tilde{\Gamma}(\pi)$ while in the second part (see Figure \ref{picturebn2}) we assume that it is.

\textbf{Part 1}.
We start by studying the case when $((\pi(n),\pi(-n))\pi,\pi)$ is not an edge in $H(B_n)$,
 this implies that in $\tilde{\Gamma}(\pi)$, $\pi(n)$ is not adjacent to $-\pi(n)$ (see Figure \ref{picturebn1}). 
From this we have that if $\pi(i)\in B(1,\pi(n))$ is and edge in $H(B_n)$, then $i<0$ implies $\pi(i)>0$.
We consider the elements adjacent to $\pi(n)$ in $\Gamma(\pi)$,
these may be divided in the two sets $U$ and $L$ as described in the previous Lemma \ref{lemman+1}. 
Let now $p\in [-n,-n+1,\ldots,n]$ be the smallest element such that $\pi(p)\in B(1,\pi(n))$,
 $\pi(p)$ may be an element in $L$ or in $U$. Note that $\pi^{-1}(B(1,\pi(n)))\subset\{\pm 1,\ldots,\pm n-1\}$ this means that there are some $\pi(k)\in B(1,\pi(n))$ such that $k<0$, in particular $p<0$.

We divide the elements in $B(1,\pi(n))$ in six subsets (see Figure \ref{picturebn1} ).
We write $L\cup \{\pi(p)\}$ as a union of four sets:
\begin{itemize}
\item We denote by $F$ the elements $\pi(i)$ such  that $i\geq0$, $0<\pi(i)<\pi(p)$;
\item We denote by $A$ the set of elements $\pi(i)$ adjacent to $\pi(n)$ such that $i>0$ and $0<\pi(i)\leq -\pi(p)$;
\item We denote by $D$ the set of elements $\pi(i)$ with $\pi(p)\geq \pi(i)\geq 0$ and $i<0$, note that $D\neq \emptyset$ because $\pi(p)\in D$;
\item  We denote by $E$ the elements $\pi(i)$ such that $i>0$ and $\pi(i)<\pi(-p)$
\end{itemize}
and we see $N\setminus\{\pi(p)\}$ as union of two:
\begin{itemize}

\item We denote by $B$ the elements $\pi(i)$ such that $i<0$ and $\pi(i)>\pi(p)$;
\item We denote by $C$ the elements $\pi(i)$ such that $i>0$ and $\pi(i)>\pi(p)$.
\end{itemize}
Given a set $X\subset \{\pm 1,\ldots,\pm n\}$, we define $-X:=\{-j: j\in X\}$ and $\pos(X):=\{|j|:j\in X\}$.

From the description we gave of the sets $L$ and $U$ we know that $F\cup A\cup D \cup E$ is a decreasing set while $\{\pi(p)\}\cup B\cup C$ is an increasing set.

We have that $|B(1,\pi(n))|>n$ and that the sets $\pi^{-1}(A)$ $,\pi^{-1}(E)$ $,\pi^{-1}(F)$ $,\pi^{-1}(C)$ are pairwise disjoint, their union has cardinality at most $n-1$ because $\pi^{-1}(A)\cup\pi^{-1}(E)\cup\pi^{-1}(F)\cup\pi^{-1}(C)\subset \{1,\ldots,n-1\}$. This implies that $|D\cup B|\geq 2$. 

We claim that:
\begin{itemize}
    \item[1.1)] $A\neq \emptyset$;
    \item[1.2)] $B\cup C\neq \emptyset$;
    \item[1.3)] $B= \emptyset$;
    \item[1.4)]$\pi(p_-)\in D$;
    \item[1.5)] $|E|\leq 1$;
    \item[1.6)] $-D=A$;
    \item[1.7)] $F=\emptyset$;
    \item[1.8)] $|C|\leq 2$;
    \item[1.9)] $E=\emptyset$.
\end{itemize}

\underline{Proof of 1.1).} Let us now consider any $j\in-\pi^{-1}(B)$, this is such that $\pi(j)<-\pi(n)<0$, in particular $j\notin \pi^{-1}(F\cup C\cup A)$. Also, we have that the elements in $B$ are in decreasing order, so the elements in $-B$ must be in decreasing order too; this implies that $|-\pi^{-1}(B)\cap \pi^{-1}(E)|\leq 1$. We know that the elements in $A$ are positive and smaller than $\pi(n)$, this implies that $-A\cap B=\emptyset$ while $-A$ and $D$ may have elements in common. We want to count the number of elements in $\pos(\pi^{-1}(L))\cup \pos(\pi^{-1}(U))$,
we know that this is: 
\begin{align*}
|\pos(\pi^{-1}(L))\cup &\pos(\pi^{-1}(U))|=\\
&=\underbrace{|A|+|B|+|D|+|C|+|F|+|E|}_{\geq n+1}-\underbrace{1}_{-B\cap E}-|D\cap -(A)|\\
&\leq n-1
\end{align*}
But $
|B(1,\pi(n))|>n
$
it follows that $|D\cap -(A)|\geq 1$ in particular $A\neq \emptyset$. 

Before proving claims 1.2-1.9, we make some remarks about the elements $\pi(p)$ may or may not be adjacent to in $\tilde{\Gamma}(\pi)$:
\begin{itemize}
\item  $\pi(p)$ is adjacent to exactly one element in $D\cup F\cup A\setminus \pi(p)$ (recall that this set is non-empty because  $A\neq \emptyset$);
\item by definition $\pi(p)$ is adjacent to $\pi(n)$;
\item If $B\cup C$ is non-empty then $\pi(p)$ is adjacent to exactly one of its elements;
\item  It is not adjacent to any element in $-B$. 

\end{itemize}

\underline{Proof of 1.2).}
We have seen that $B(1,\pi(i))\cap B(1,\pi(n))$ must have cardinality least two. One of these must be in $B\cup C$ which is therefore non-empty. Notice also that this implies that $B(1,\pi(p))\cup B(1,\pi(n))=\{\pm 1,\ldots,\pm n\}$.

\underline{Proof of 1.3)}. Suppose that $\pi(b)\in B$, we have that $-\pi(b)\notin B(1,\pi(p))$, this implies that $\pi(b)$ must be adjacent to $\pi(n)$. We see that $-\pi(b)<-\pi(p)$ so $-\pi(b)\in E$ but $-b<-p$ which is in contradiction with the monotonicity of $L$, we conclude that $B$ must be empty.

%this would give an element with degree smaller or equal to $n$ unless $|B|=1$ and $-B\subset E$. It is easy to notice that in this case we would have that $-\pi(p)\not in A$ and no loops at $\pi(p)$ in $\Gamma(\pi)$ and obtain\[d(\pi(p))\leq 2n-d(\pi(n))+\underbrace{1}_{D}+\underbrace{1}_{B\cup F}-\underbrace{1}_{-\pi(p)}\leq n\quad.\]
\underline{Proof of 1.4).}.We just found that $B\neq \emptyset$. Because $B\cup C\neq \emptyset$ we conclude that which implies $|C|\geq 1$.
We also know that $|B\cup D|\geq 2$, it follows that $|D|\geq 2$, in particular $\pi(p_-)\in D$.

\underline{Proof of 1.5)}. 
We know that the elements in $E$ are in descending order, as a consequence the elements in $-E$ must be in ascending order. We expect the elements in $-E$ to be on the right of $\pi(p)$, this implies that they don't belogn to $B(1,\pi(n))$. If $E$ had more than one element, than there would be a $\pi(e)\in E$ such that $-\pi(e)$ is not adjacent to $\pi(p)$, that would give one element not adjacent to $\pi(p)$ nor to $\pi(n)$, from now on we assume $|E|\leq 1$.

\underline{Proof of 1.6)}. Now we study the elements in $-D$, none of these is adjacent to $\pi(p)$ so they must be all adjacent to $\pi(n)$. Given $\pi(d)\in D$, we have that $-\pi(p)\leq -\pi(d)<0$, we conclude that $-D\subset A$. 
Let $\pi(a)\in A\setminus -D$, we have that $0>\pi(a)>-\pi(p)$ and  $-\pi(p)$ is the rightmost element in $A$. The elements in $A$ are in ascending order, therefore there is exactly one element in $A$ adjacent to $-\pi(p)$ in $\tilde{\Gamma(\pi)}$, this must be $-\pi(p_-)$. Therefore $\pi(a)$ is not adjacent to $-\pi(p)$ and conversely $-\pi(a)$ is not adjacent to $\pi(p)$, also $\pi(a)\notin D$ and so $\pi(a)\notin B(1,\pi(n))$. We found an element not in $B(1,\pi(n))\cap B(1,\pi(p))$, this leads to a contradiction.

\underline{Proof of 1.7)}. We take now $\pi(f)\in F$, we have that for any $\pi(a)\in A$, $f<a$, therefore for any $-a\in -A$ we have $-f>-a$ and $\pi(f)>0>\pi(a)$. We know that $\pi(p)$ is adjacent to one element in $D$, this implies that it is not adjacent to any element in $-F$, moreover $-F\not\subset B(1,\pi(n))$ we conclude that $F=\emptyset$.

\underline{Proof of 1.8)}. We recall that $|D|\geq 2$ and that $\pi(p)$ must be adjacent to all the elements in $-C$, equivalently $-\pi(p)$ is adjacent to the elements in $C$. We denote $C=\{\pi(c_1),\ldots,\pi(c_k)\} $ where
 $0<c_1<\ldots<c_k$ and $0<\pi(c_1)<\ldots<\pi(c_k)$. We deduce that, except from $c_k$ that may be bigger than $-p$, all the $c_j$s  are in $[-p_-,-p]$. Consider $\pi(p_-)$ in $\tilde{\Gamma}(\pi)$, this is adjacent to at most three elements among the ones adjacent to $\pi(n)$: two of these are in $A\cup (-A)$ and one is in $C$.
This means that there is at most one element in $\{B(1,\pi(p_-))\cup B(1,\pi(n))\}^c$. 
According to what we said about $-C$ we have that $\pi(p')$ is adjacent to at most one element in it, we deduce that $|-C|\leq 2$.
%  We obtain that\[|\{\text{elements adjacent to }\pi(p')\}|=2n-(n+1)+3-(C|-1)\geq n+1\Rightarrow\quad|C|\leq 2.\]

\underline{Proof of $1.9)$}. We notice that $n+1=2|D|+|E|+|C|$, because we assumed that $n$ is odd
 we must have $|E|+|C|\equiv 0$ $mod(2)$ . Only two situations may occur: $|E|=|C|=1$ or $E=\emptyset$ and $|C|=2$.
Assume that $|E|=|C|=1$ and let $\{\pi(e)\}= E$.
The elements that are adjacent to both $\pi(e)$ and $\pi(n)$ in $\tilde{\Gamma}(\pi)$ are two: one in $A$ and the other in $C$. We deduce that there can't be vertices in $\tilde{\Gamma}(\pi(n))$ that are not adjacent to both $\pi(e)$ and $\pi(n)$.
We notice that $-\pi(e)<\pi(p)<\pi(e)$ and $-e<p<e$ so $\pi(e)$ and $-\pi(e)$ are not adjacent and $-\pi(e)$ is not adjacent to $\pi(n)$ as well because $B=\emptyset$. We found an element in $B(1,\pi(e))\cap B(1,\pi(n))$, this leads to a contradiction. This implies that $|E|=0$ and $|C|=2$.

\vspace{3pt}

The only case left is when $E=\emptyset$ and $|C|=2$. Let $\pi(c_1)<\pi(c_2)$ the elements in $C$, we notice that $\pi(c_2)$ in $\tilde{\Gamma}(\pi)$ is adjacent to at most one element in $A$ and to $\pi(c_1)$, we deduce that any element is adjacent to either $\pi(n)$ either to $\pi(c_2)$. We notice that the elements in $-C$
 are not adjacent to $\pi(n)$ nor to $\pi(c_1)$, we conclude that $\pi(p)\notin D$.
 
 What we conclude from this first part of the proof is that if $\pi(n)$ and $-\pi(n)$ are not adjacent in $\tilde{\Gamma}(n)$, then the statement holds.

\textbf{Part 2}. From now on we assume that $\pi(n)$ and $-\pi(n)$ are adjacent in $\tilde{\Gamma}(\pi)$. This means that there are no elements $\pi(i)$ such that $-\pi(n)<\pi(i)<\pi(n)$, as a consequence we have $\pi(n)=1$.

We consider the set of elements that are adjacent to $1$ in $\tilde{\Gamma}(\pi)$ and divide them in three sets (see Picture \ref{picturebn2}):
\begin{itemize}
\item $A$ is the set of negative vertices $\pi(a)\neq -1$ adjacent to $1$. The elements in this set are in descending order;
\item $B$ is the set of positive elements $\pi(b)$ adjacent to $1$. The elements in this set are in ascending order;
\item $C=\{\pi(-n)=-1\}$.
\end{itemize}
We are assuming that in $\tilde{\Gamma}(\pi)$ at least $n+1$ edges start from $1$, this implies that $|A|+|B|\geq n$. Because $A\subset\{-2,\ldots,-n\}$ and $B\subset \{2,\ldots,n\}$ we conclude that both $A$ are non-empty. We claim the following:
\begin{itemize}
    \item[2.1)] $|A\cap -B|=1$;
    \item[2.2)] $(-A)\cap B=\{2\}$;
\end{itemize}

\underline{Proof of 2.1)}. We know that $A$ and $B$ are disjoint sets, yet $A$ and $-B$ may intersect. We show that such an intersection must be non-empty:
\[
\underbrace{|A|+|B|}_{\geq n}+\underbrace{|-A|+|-B|}_{\geq n}-|A\cap -B|-|-A\cap B|\leq 2(n-1).\quad \Rightarrow  A\cap B\neq \emptyset.
\]
We already noticed that the elements in $A$ are in descending order and the elements in $B$ are in ascending order this implies that $|A\cap -B|\leq 1$. We conclude that $|A\cap B|=1$ and $|A|+|B|-1=n-1$.

\underline{Proof of 2.2)}. Let $\pi(p)=-A\cap B$, we notice that it is adjacent to both $\pi(n)=1$ and $-\pi(n)=-1$. This means that there are no $\pi(k)$ with $|k|<n$ such that $-1<\pi(k)<\pi(p)$, implying $\pi(p)=2$.

We consider the rightmost element in $A\cup B$, this can be either $\pi(p)=-A\cap B$ either $-\pi(p)=A\cap -B$, we assume that in our case the leftmost element is $\pi(p)\in B$.

To end the proof we consider $\pi(b)$ as the rightmost element in $B$ and we make some remarks about the elements adjacent and not adjacent to $\pi(p)$ and $\pi(b)$:
\begin{itemize}
\item $\pi(p)$ is adjacent to at most three elements in $B(1,\pi(n))$ and to at most one element in $-A$, so is not adjacent to at least $n-2+|A|-2$ elements ;
\item $\pi(b)$ is adjacent to at most one element in $B$ and potentially all the elements in $A$;
\item if $B_-$ is the set of $\pi(k)\in B$ such that $k<0$, then $\pi(p)$ is not adjacent to $-B_-$;
\item if $B_+$ is the set of elements $\pi(k)\in B$ with $k>0$ and $|B_+|\geq 2$ then $\pi(b)$ is not adjacent to the 
elements in $-B_+$.
\end{itemize}
From these remarks we deduce that if $|A|\geq 4$ then $|B(1,\pi(p))|\leq n$. 
Recall that we are in case $n\geq 7$, so $|A|+|B|\geq 7$. If $|A|=1,2$ we have that $|B|\geq 5$ and this implies that either $|B_+|\geq 2$ or $|B_-|\geq 4$. In the first case we have that $|B(1,\pi(b))|\leq n$ in the second case $|B(1,\pi(p))|\leq n$. Only the case $|A|=3$ is left, we see that if in this case $|B_+|\geq 3$ then $B(1,\pi(b))$ has less than $n+1$ elements. Otherwise $|B_+|<3$ so $\pi(p)$ is not adjacent to any element in $-A$, any element in $-B_{-}$ and $n-2$ elements in $B(1,\pi(n))$. This concludes the study of this second part of the proof. Notice that if we had had $-\pi(p)$ as the leftmost element in $-B$ we could have followed the same reasoning using $\pi(b)$ as the rightmost element in $A$ and changing the roles of $A$ and $B$.
\end{proof}

Now we are ready to prove the main result of this section:

\begin{proof}[Proof of Theorem \ref{teobn}]
The total degree of an element $\pi\in B_n$ is the number of edges (including loops) that appear in $\Gamma(\pi)$.

We start studying the case $n=5$, this is the base step of an induction. The fact that the maximal degree in $H(B_5)$ is $16$ can be easily checked using the functions bruhat\textunderscore upper\textunderscore
covers and bruhat\textunderscore lower \textunderscore covers on SageMath (see \label{sage}). As an example of element that actually reaches this maximum degree we can take $[1,2,-5,-4,-3]$.

 Let now $\pi\in B_{n}$, $n>5$ and assume that the Theorem is true for any element in $B_{n-1}$.
 We consider a vertex $a$ in $\mathcal{V}(\Gamma(\pi))$ with no more than $n+1$ edges if $n$ is even and no more than $n$ edges if $n$ is odd. We define $\hat{\Gamma}(\pi)$ as the graph obtained erasing from $\Gamma(\pi)$ the node $a$ and its edges. We denote by $\pi'$ the element of $B_{n-1}$ obtained deleting $a$ from $\pi$ and subtracting $1$ to the values larger than $a$. 
We compare now the two graphs $\hat{\Gamma}(\pi)$ and $\Gamma(\pi')$ and notice that the edges of the second one include the ones of the first. Therefore we have that
\[
|\{\text{ edges in }(\Gamma(\pi'))\}|\geq |\{\text{ edges is }({\Gamma(\hat{\pi})})\}|.
\]
We use the induction hypothesis and split it into two cases:
\begin{itemize}
\item $n$ even
\[
e(\Gamma)(\pi)=e(\hat{\Gamma}(\pi))+d(a)\leq \lfloor \frac{(n-1)^2}{2}\rfloor+(n-1)-1+n+1=\lfloor \frac{n^2}{2}\rfloor+n-1;
\]
\item $n$ odd
\[
e(\Gamma)(\pi)=e(\hat{\Gamma}(\pi))+d(a)\leq \lfloor \frac{(n-1)^2}{2}\rfloor+(n-1)-1+n=\lfloor \frac{n^2}{2}\rfloor+n-1;
\]
\end{itemize}
To show the equality we exhibit ad element in $B_{n}$, $n\geq 5$ that satisfies the equality:
\[
[1,2,\ldots,m,-n,-(n-1),\ldots,-(m+1)].
\]
\end{proof}

Now we have to study the cases $B_2$, $B_3$ and $B_{4}$. The following Propositions hold:
\begin{prop}
The maximal degree of a vertex in the Hasse diagram of the Bruhat order of $B_n$ is $4(n-1)$ for $n\leq 5$.
\end{prop}

\begin{proof}
It is easy to verify that the Proposition holds for $B_2$.
For cases $B_3$ and $B_4$ it is possible to compute de maximum degree of the Bruhat order using the functions bruhat \textunderscore upper \textunderscore covers and bruhat \textunderscore lower \textunderscore covers on SageMath (see \cite{sage}).

\end{proof}

\subsection{On the degree of the Hasse diagram of Dn}
\label{dmaxdn}
In this section we study the elements of maximum degree in the Hasse graph of the Bruhat order of $D_n$. Our proof follows the lines of the one in the previous section.

We start therefore with a description of the covering relations.

\begin{prop}\label{propdn}
Given two elements $\pi,\sigma \in D_n$ with $\pi>\sigma$ there is an edge in $H(D_n)$ between $\pi$ and $\sigma$ if $\pi=s\sigma $ with $s$ of the following kind:
\begin{itemize}
\item[1)] $s=(\pi(i),\pi(j))(-\pi(i),-\pi(j))$ with $0<i<j$, $\pi(i)<\pi(j)$ such that for any $i<k<j$ $	\pi(k)\notin [\pi(i),\pi(j)]$; 
\item[2)] $s=(\pi(i),-\pi(j))(\pi(j),-\pi(i))$ with $i,j>0$, $sign(\pi(i))=sign(-\pi(j))$, $-\pi(j)<\pi(i)$ and such that  given $k\in [-j,i]$, $\pi(k)>\pi(i)$ or $\pi(k)<-\pi(j)$.
\item[3)] $s=(\pi(i),-\pi(j))(-\pi(i),\pi(j))$ with $j>i>0$, $\pi(i)<\pi(j)<0$ and  given $k\in [-i,j]$, $\pi(k)\notin [\pi(i),-\pi(j)]$.
\item[4)] $s=(\pi(i),-\pi(j))(-\pi(i),\pi(j))$ with $j>i>0$, $\pi(j)>\pi(i)>0$ and  given $k\in [-i,j]\setminus\{i\}$, $\pi(k)\notin [-\pi(i),\pi(j)]$.
\end{itemize}
\end{prop} 

\begin{obs}
As $D_n$ is a subgroup of $B_n$ it is in particular a subset of it. The Bruhat order on $D_n$ is the one induced by the Bruhat order of $B_n$. In particular it is an induced subposet of $A_{2n-1}=S_{[-n,\ldots,n]}$ with its Bruhat order.
\end{obs}

We would like to associate to every element $\pi\in D_n$ a graph $\Gamma(\pi)$ that encodes all the information about its covers and coverings in $H(D_n)$. We define these graphs as follows:
\begin{itemize}
\item they have $2n$ vertices labeled with the elements of the set $\{\pm 1,\ldots,\pm n\} $;
\item there is an edge from $i$ to $j$ and one from $-i$ to $-j$ if $(i,j)(-i,-j)\pi$ is adjacent to $\pi$ in $H(D_n)$;
\item there is an edge between $i$ and $j$ and one between $-i$ and $-j$ if $(i,-j)(-i,j)\pi$ is adjacent to $\pi$ in $H(D_n)$.
\end{itemize}
Notice that according to Proposition \ref{propdn} there are no edges between $i$ and $-i$ for any $i\in \{1,\ldots,n\}$; it follows that the number of elements adjacent to $\pi$ in $H(B_n)$ is half of the number of edges in $\Gamma(\pi)$.

\begin{example}
In Figure \ref{gammadn} an example of $\Gamma(\pi)$ can be found, here we have  $\pi=[2,-3,-4,1] $ $\in D_4$. The coverings of $\pi$ are $(2,4)(-2,-4)\pi$, $(3,4)(-3,-4)\pi$; $\pi$ covers $(1,2)(-1,-2)\pi$ $(1,-2)(-1,2)\pi$ $(1,4)(-1,-4)\pi$ and $(2,3)(-2,-3)\pi$.
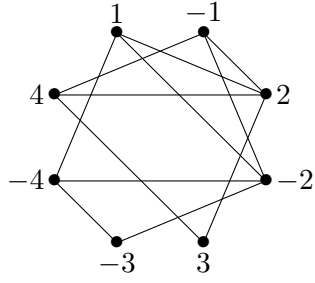
\begin{figure}
\centering
\begin{tikzpicture}[scale=1]
     \draw (-1.39,0.57) node[anchor=east]  {$4$};
     \draw (-1.39,-0.57) node[anchor=east]  {$-4$};
     \draw (1.39,0.57) node[anchor=west]  {$2$};
     \draw (1.39,-0.57) node[anchor=west]  {$-2$};
     \draw (0.57,-1.39) node[anchor=north]  {$3$};
     \draw (-0.57,-1.39) node[anchor=north]  {$-3$};
     \draw (0.57,1.39) node[anchor=south]  {$-1$};
     \draw (-0.57,1.39) node[anchor=south]  {$1$};
\foreach \Point in {(0.57,1.39),(-0.57,1.39),(-0.57,-1.39),(0.57,-1.39),(1.39,-0.57),(1.39,0.57),(-1.39,-0.57),(-1.39,0.57)}{
    \node at \Point {\textbullet};
}
\path
(0.57,1.39) edge [] node {} (1.39,0.57)
(0.57,1.39) edge [] node {} (1.39,-0.57)
(0.57,1.39) edge [] node {} (-1.39,0.57)
(-0.57,1.39) edge [] node {} (1.39,-0.57)
(-0.57,1.39) edge [] node {} (1.39,0.57)
(-0.57,1.39) edge [] node {} (-1.39,-0.57)
(1.39,0.57) edge [] node {} (0.57,-1.39)
(1.39,0.57) edge [] node {} (-1.39,0.57)
(1.39,-0.57) edge [] node {} (-0.57,-1.39)
(1.39,-0.57) edge [] node {} (-1.39,-0.57)
(0.57,-1.39) edge [] node {} (-1.39,0.57)
(-0.57,-1.39) edge [] node {} (-1.39,-0.57);
\end{tikzpicture}
\caption{${\Gamma}([2,-3,-4,1])$}\label{gammadn}
\end{figure}

\end{example}

We begin by proving the analogous of \cite[Lemma 3.2]{adinroich} (or Lemma \ref{lemman+1} of this thesis) for $D_n$.

\begin{prop}
For any $\pi\in D_n$ there is a vertex in $\Gamma(\pi)$ of degree at most $n+1$.
\end{prop}

\begin{proof}
We consider $\pi(n)\in V(\Gamma(\pi))$ and assume that $\pi(n)>0$; we suppose by contradiction that there are at least $n+2$ edges stemming from every vertex of $\Gamma(\pi)$.
We denote by $U$ ( resp. $L$) the set of elements adjacent to $\pi(n)$ in $\Gamma(\pi)$ and bigger (resp. smaller) than it.
We fix $\pi(p)$ as the vertex adjacent to $\pi(n)$ with $p$ minimal.
We want to study the number of vertices that are adjacent to both $\pi(p)$ and $\pi(n)$. We denote by $\pi(p_+)$ the smallest element in $U\setminus \{\pi(p)\}$, and by $\pi(p_-)$ the biggest element in $L\setminus \{\pi(p)\}$. The following holds:
\begin{itemize}
    \item $B(1,\pi(p))\cap U$ is empty if $U\setminus {\pi(p)}=\emptyset$ and is the singleton $\{\pi(p_+)\}$ otherwise;
    \item $\pi(p)$ is adjacent to $\pi(p_-)$ if $L\setminus \{\pi(p)\}\neq \emptyset$, it than may be adjacent to $-\pi(p_-)$. We conclude that $|B(1,\pi(p))\cap U|\leq 2$.
\end{itemize}
It follows that there are at most three vertices adjacent to both $\pi(p)$ and $\pi(n)$. As a consequence we have this inequality:
$$
|B(1,\pi(p))|\leq 2n-(n+2)+3=n+1.
$$
This concludes the proof for the case $\pi(n)>0$. Because we haven't used the fact the set $\{i\in\{1,\ldots,n\}|\pi(i)<0\}$ has even cardinality the same proof stands for case $\pi(n)<0$.
\end{proof}

Again, if $n$ is odd, more can be said:

\begin{prop}
If $n>3$ and odd, then there is a vertex in $\Gamma(n)$ of degree at most $n$.
\end{prop}

\begin{proof}
We assume that $\pi(n)>0$, also we assume that $|B(1,\pi(i))|\geq n+1$ in $\Gamma(\pi)$ for any $i\in \{\pm 1,\ldots,\pm n\}$. 

First we notice the following fact that is of great utility to tackle the proof.
Let $\pi(i)\neq \pi(n)$ be an element in $\{\pm1,\ldots,\pm\}$ adjacent to exactly $k$ elements in $B(1,\pi(n))$. Because $\Gamma(\pi)$ has $2n$ vertices we have that:
\[
(B(1,\pi(n))\cup B(1,\pi(i)))^c\leq 2n-2(n+1)+k=k-2.
\]  
In particular we  have that $B(1,\pi(n))\cap B(1,\pi(i))$ has at least two elements, and that there are at most $k-2$ vertices in $\Gamma(\pi)$ that are not adjacent to both $\pi(i)$ and $\pi(n)$.
We go on with the rest of the proof which is divided in three cases:
\begin{itemize}
\item[1] there are no elements $\pi(k)$ adjacent to $\pi(n)$ such that $\pi(k)<0$ and $k<0$ (see Picture \ref{picturedn1});
\item[2] there are no elements $\pi(j)$ adjacent $\pi(n)$ such that $0<\pi(j)<\pi(n)$ and $j>0$ and there is at least one element $\pi(k)$ such that $\pi(k)<0$ and $j<0$ (see Picture \ref{picturedn2}); 
\item[3] among the elements adjacent to $\pi(n)$ there is at least one $\pi(k)<0$ with $k<0$ and at least one $0<\pi(j)<\pi(n)$ with $j>0$  (see Picture \ref{picturedn3}).
\end{itemize}

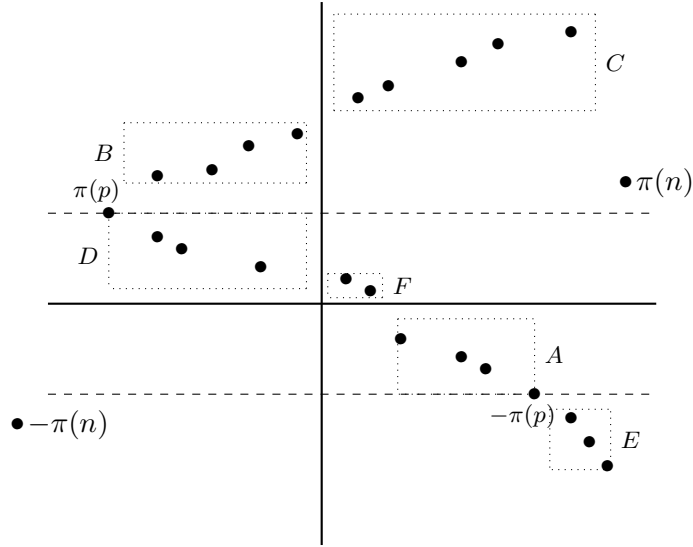
\begin{figure}
  \centering
   \begin{tikzpicture}[scale=.8]
      \draw (4.75,-2.25) node[anchor=west]  {\footnotesize $E$};
      \draw (3.5,-0.85) node[anchor=west]  {\footnotesize $A$};
      \draw (1,0.3) node[anchor=west]  {\footnotesize $F$};
      \draw (4.5,4) node[anchor=west]  {\footnotesize $C$}; 
      \draw (-3.25,2.5) node[anchor=east]  {\footnotesize $B$}; 
      \draw (-3.5,.8) node[anchor=east]  {\footnotesize $D$}; 
      \draw (-5,-2) node[anchor=west]  {$-\pi(n)$};  %   
      \draw (5,2) node[anchor=west]  {$\pi(n)$};  %
      \draw (-3.7,1.5) node[anchor=south]  {\footnotesize $\pi(p)$};
      \draw (3.3,-1.5) node[anchor=north]  {\footnotesize $-\pi(p)$};
\foreach \Point in {(5,2),%\pi(n)
                       (-5,-2),%\pi(n)
                              (3.5,-1.5), %-\pi(p)
                              (-3.5,1.5), %\pi(p)
               %elementi in D
               (-2.7,1.1),(-2.3,0.9),(-1,0.6), 
              % elementi in F 
               (0.4,0.4),(0.8,0.2),
               %elementi in A
               (2.7,-1.1),(2.3,-0.9),(1.3,-0.6),
               %elementi in E
               (4.1,-1.9),(4.4,-2.3),(4.7,-2.7),               
               %elementi in B
               (-.4,2.8),(-1.2,2.6),(-1.8,2.2),(-2.7,2.1),
               % elementi in C
               (0.6,3.4),(1.1,3.6),(2.3,4),(2.9,4.3),(4.1,4.5)
               %                               (-1,0.75), %più basso elemento in b
%                              (-3.5,2.25), %più alto elemento in b
%(1,0),(2,0),(1,1),(2,1),(2,2)
}{
     \node at \Point {\textbullet};
}
     \draw[thick] (-4.5,0) -- +(10,0); %yes
     \draw[thick] (0,-4) -- +(0,9); %yes
     \draw[dashed] (-4.5,1.5) -- +(10,0); %yes
     \draw[dashed] (-4.5,-1.5) -- +(10,0); %yes
% rettangolo che delimita D
     \draw[dotted] (-3.5,1.5) -- +(3.25,0);
     \draw[dotted] (-.25,.25) -- +(-3.25,0);
     \draw[dotted] (-.25,.25) -- +(0,1.25);
     \draw[dotted] (-3.5,1.5) -- +(0,-1.25);
% rettangolo che delimita B
     \draw[dotted] (-3.25,3) -- +(3,0);
     \draw[dotted] (-.25,2) -- +(-3,0);
     \draw[dotted] (-.25,2) -- +(0,1);
     \draw[dotted] (-3.25,3) -- +(0,-1);
% rettangolo che delimita A 
     \draw[dotted] (3.5,-1.5) -- +(-2.25,0);
     \draw[dotted] (1.25,-.25) -- +(2.25,0);
     \draw[dotted] (1.25,-.25) -- +(0,-1.25);
     \draw[dotted] (3.5,-1.5) -- +(0,1.25);
% rettangolo che delimita E 
     \draw[dotted] (3.75,-1.75) -- +(1,0);
     \draw[dotted] (4.75,-2.75) -- +(-1,0);
     \draw[dotted] (4.75,-2.75) -- +(0,1);
     \draw[dotted] (3.75,-1.75) -- +(0,-1);
% rettangolo che delimita F 
     \draw[dotted] (0.1,.5) -- +(.9,0);
     \draw[dotted] (1,0.1) -- +(-.9,0);
     \draw[dotted] (1,0.1) -- +(0,.4);
     \draw[dotted] (0.1,.5) -- +(0,-.4);
% rettangolo che delimita C 
     \draw[dotted] (0.2,3.2) -- +(4.3,0);
     \draw[dotted] (4.5,4.8) -- +(-4.3,0);
     \draw[dotted] (4.5,4.8) -- +(0,-1.6);
     \draw[dotted] (0.2,3.2) -- +(0,1.6);
   \end{tikzpicture}
\caption{1st case}
\label{picturedn1}
\end{figure}

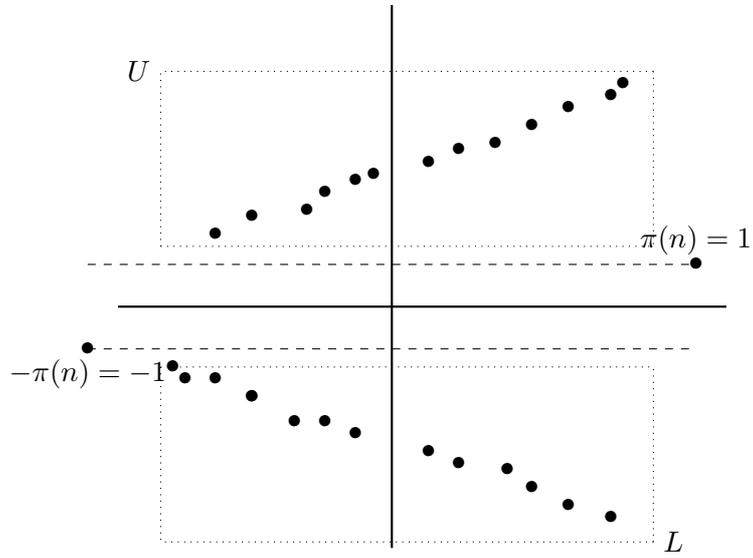
\begin{figure}
  \centering
   \begin{tikzpicture}[scale=.8]
            \draw (4.3,-3.9) node[anchor=west]  {$L$};  %
            \draw (-3.8,3.9) node[anchor=east]  {$U$};  % 
            \draw (-5,-.7) node[anchor=north]  {$-\pi(n)=-1$};  %   
      \draw (5,.7) node[anchor=south]  {$\pi(n)=1$};  %
\foreach \Point in {(5,.7),%\pi(n)
                    (-5,-.7),%\pi(n)
               % elementi in -C
               (-0.6,2.1),(-1.1,1.9),(-2.3,1.5),             			   (-2.9,1.2),
               %(-3.6,1),
               % elementi in C
               (0.6,2.4),(1.1,2.6),(1.7,2.7),     			   (2.3,3),			                       			   (2.9,3.3),(3.6,3.5),(3.8,3.7),
               %elementi in B
               (-.3,2.2),(-1.4,1.6),
               %elementi in -A
               (-3.4,-1.2),(-2.3,-1.5),(-1.6,-1.9),
               % elementi in -C
               (-0.6,-2.1),(-1.1,-1.9),(-2.3,-1.5),             			   (-2.9,-1.2),(-3.6,-1),
               % elementi in C
               (0.6,-2.4),(1.1,-2.6),					           (1.9,-2.7),(2.3,-3),			                       			   (2.9,-3.3),(3.6,-3.5)
}{
     \node at \Point {\textbullet};
}
     \draw[thick] (-4.5,0) -- +(10,0); %yes
     \draw[thick] (0,-4) -- +(0,9); %yes
     \draw[dashed] (-5,.7) -- +(10,0); %yes
     \draw[dashed] (-5,-.7) -- +(10,0); %yes
% rettangolo che delimita U 
     \draw[dotted] (4.3,3.9) -- +(-8.1,0);
     \draw[dotted] (-3.8,1) -- +(+8.1,0);
     \draw[dotted] (4.3,3.9) -- +(0,-2.9);
     \draw[dotted] (-3.8,1) -- +(0,+2.9);
% rettangolo che delimita L 
     \draw[dotted] (4.3,-3.9) -- +(-8.1,0);
     \draw[dotted] (-3.8,-1) -- +(+8.1,0);
     \draw[dotted] (4.3,-3.9) -- +(0,+2.9);
     \draw[dotted] (-3.8,-1) -- +(0,-2.9);
   \end{tikzpicture}
\caption{2nd case}
\label{picturedn2}
\end{figure}

\begin{figure}
  \centering
   \begin{tikzpicture}[scale=.8]
      \draw (4.75,-2.25) node[anchor=west]  {\footnotesize $E$};
      \draw (3.5,-0.85) node[anchor=west]  {\footnotesize $A$};
      \draw (1,0.65) node[anchor=west]  {\footnotesize $F$};
      \draw (-1,-0.65) node[anchor=east]  {\footnotesize $G$};
      \draw (4.5,4) node[anchor=west]  {\footnotesize $C$}; 
      \draw (-3.25,2.5) node[anchor=east]  {\footnotesize $B$}; 
      \draw (-3.5,.8) node[anchor=east]  {\footnotesize $D$}; 
      \draw (-5,-2) node[anchor=west]  {$-\pi(n)$};  %   
      \draw (5,2) node[anchor=west]  {$\pi(n)$};  %
      \draw (-3.7,1.5) node[anchor=south]  {\footnotesize $\pi(p)$};
      \draw (3.3,-1.5) node[anchor=north]  {\footnotesize $-\pi(p)$};
\foreach \Point in {(5,2),%\pi(n)
                       (-5,-2),%\pi(n)
                              (3.5,-1.5), %-\pi(p)
                              (-3.5,1.5), %\pi(p)
               %elementi in D
               (-2.7,1.2),(-2.3,1), 
              % elementi in G 
               (-0.65,-0.65),
              % elementi in F 
               (0.65,0.65),
               %elementi in A
               (2.7,-1.2),(2.3,-1),
               %elementi in E
               (4.1,-1.9),(4.4,-2.3),(4.7,-2.7),               
               %elementi in B
               (-.4,2.8),(-1.2,2.6),(-1.8,2.2),(-2.7,2.1),
               % elementi in C
               (0.6,3.4),(1.1,3.6),(2.3,4),(2.9,4.3),(4.1,4.5)
               %                               (-1,0.75), %più basso elemento in b
%                              (-3.5,2.25), %più alto elemento in b
%(1,0),(2,0),(1,1),(2,1),(2,2)
}{
     \node at \Point {\textbullet};
}
     \draw[thick] (-4.5,0) -- +(10,0); %yes
     \draw[thick] (0,-4) -- +(0,9); %yes
     \draw[dashed] (-4.5,1.5) -- +(10,0); %yes
     \draw[dashed] (-4.5,-1.5) -- +(10,0); %yes
% rettangolo che delimita D
     \draw[dotted] (-3.5,1.5) -- +(2.5,0);
     \draw[dotted] (-1,.7) -- +(-2.5,0);
     \draw[dotted] (-1,.7) -- +(0,.8);
     \draw[dotted] (-3.5,1.5) -- +(0,-.8);
% rettangolo che delimita B
     \draw[dotted] (-3.25,3) -- +(3,0);
     \draw[dotted] (-.25,2) -- +(-3,0);
     \draw[dotted] (-.25,2) -- +(0,1);
     \draw[dotted] (-3.25,3) -- +(0,-1);
% rettangolo che delimita A 
     \draw[dotted] (3.5,-1.5) -- +(-2,0);
     \draw[dotted] (1.5,-.7) -- +(2,0);
     \draw[dotted] (1.5,-.7) -- +(0,-.8);
     \draw[dotted] (3.5,-1.5) -- +(0,.8);
% rettangolo che delimita E 
     \draw[dotted] (3.75,-1.75) -- +(1,0);
     \draw[dotted] (4.75,-2.75) -- +(-1,0);
     \draw[dotted] (4.75,-2.75) -- +(0,1);
     \draw[dotted] (3.75,-1.75) -- +(0,-1);
% rettangolo che delimita F 
     \draw[dotted] (0.3,.3) -- +(0,0.7);
     \draw[dotted] (1,1) -- +(-.0,-0.7);
     \draw[dotted] (0.3,.3) -- +(.7,0);
     \draw[dotted] (1,1) -- +(-.7,0);
% rettangolo che delimita G
     \draw[dotted] (-0.3,-.3) -- +(0,-0.7);
     \draw[dotted] (-1,-1) -- +(-.0,0.7);
     \draw[dotted] (-0.3,-.3) -- +(-.7,0);
     \draw[dotted] (-1,-1) -- +(.7,0);
% rettangolo che delimita C 
     \draw[dotted] (0.2,3.2) -- +(4.3,0);
     \draw[dotted] (4.5,4.8) -- +(-4.3,0);
     \draw[dotted] (4.5,4.8) -- +(0,-1.6);
     \draw[dotted] (0.2,3.2) -- +(0,1.6);
   \end{tikzpicture}
\caption{3rd case}
\label{picturedn3}
\end{figure}
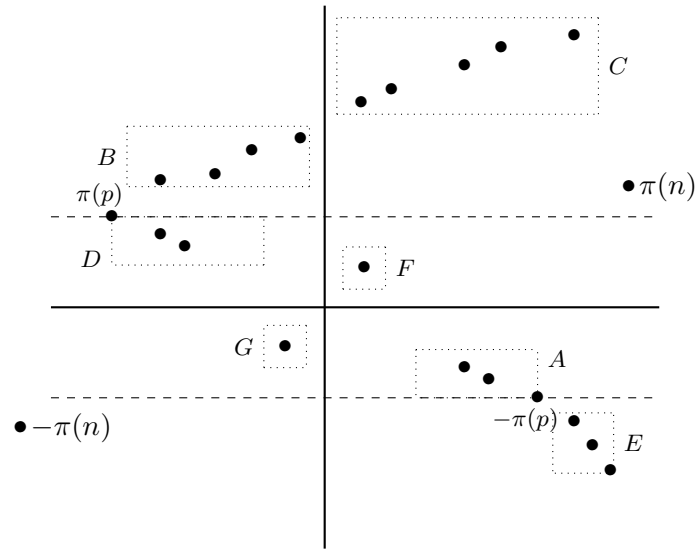

\textbf{1st case}.
Take $\pi(p)$ as the element in $B(1,\pi(n))$ with the smallest $p$, because $|B(1,\pi(n))|>n$ and we are in the 1st case, we see that $p<0$ and $\pi(p)>0$.
We divide the elements adjacent to $\pi(n)$ in seven sets (see Figure \ref{picturedn1}):
\begin{itemize}
\item $A:=\{\pi(i)\in \{-n,\ldots,\hat{0},\ldots,n\}|i>0,\, -\pi(p)\leq \pi(i)<0\}$, notice that a priori $-\pi(p)$ may not be adjacent to $\pi(n)$;
\item $B:=\{\pi(i)\in \{-n,\ldots,\hat{0},\ldots,n\}|\pi(i)>\pi(p),\,i<0\}$;
\item $C:=\{\pi(i)\in \{-n,\ldots,\hat{0},\ldots,n\}|\pi(i)>\pi(p),\, i>0\}$;
\item $D:=\{\pi(i)\in \{-n,\ldots,\hat{0},\ldots,n\}|\pi(i)\leq \pi(p),\,i<0\}$ , in particular we notice that $\pi(p)\in D$;
\item $E:=\{\pi(i)\in \{-n,\ldots,\hat{0},\ldots,n\}|\pi(i)<0,\,\pi(i)<-\pi(p)\}$;
\item $F:=\{\pi(i)\in \{-n,\ldots,\hat{0},\ldots,n\}|0\leq \pi(i)\leq \pi(n),\,i> 0\}$; 
\item $G:=\{\pi(i)\in \{-n,\ldots,\hat{0},\ldots,n\}|\pi(i)<0,\,i<0\}$.
\end{itemize}
For this first case $G=\emptyset$.
Given a set $X\subset \{\pm 1,\ldots,\pm n\}$ we define $-X:=\{-x|x\in X\}$.
We notice that $|C\cup F\cup A\cup E|\leq n-1$, from this we obtain that $|D\cup B|\geq 2$. We study the number of elements in $B(1,\pi(p))\cap B(1,\pi(n))$ in $\Gamma(\pi)$:
\begin{itemize}
\item $\pi(p)$ is adjacent to the smallest element in $B\cup C$ if this is not empty, we denote this element $\pi(p_+)$;
\item $\pi(p)$ is adjacent to the biggest element in $(D\cup F \cup A \cup E)\setminus \{\pi(p)\}$ if this is non-empty, we refer to it as to $\pi(p_-)$;
\item If $-\pi(p_-)\in B(1,\pi(n))$ then $\pi(p)$ may be adjacent to it.
\end{itemize} 

Through this remark we study the following three subcases:
\begin{itemize}
\item[1.a)] This is the case when $B(1,\pi(p))\cap B(1,\pi(n))$ have two elements: $\pi(p_-)$ and $\pi(p)_+$;  
\item[1.b)] this is the case when $B(1,\pi(p))\cap B(1,\pi(n))$ have three elements: $\pi(p_-)$, $\pi(p)_+$ and $-\pi(p_-)$;  
\item[1.c)] this is the case when $B(1,\pi(p))\cap B(1,\pi(n))$ have two elements: $\pi(p_-)$ and $-\pi(p_-)$. Here we have $B\cup C=\emptyset$.
\end{itemize}
Note that we are not considering the case of $(D\cup F \cup A \cup E)\setminus \{\pi(p)\}=\emptyset$, because there must be at least two elements in $B(1,\pi(p))\cap B(1,\pi(n))$.
\textbf{Subcase 1.a}.
We assume that $|B(\pi(p))\cap B(\pi(n))|=2$. In this case we need that any element not adjacent to $\pi(p)$ is in $B(\pi(n))\setminus\{\pi(p_-),\pi(p_+)\}$ and viceversa. We claim:
\begin{itemize}
    \item[1a.1]$B= \emptyset$; 
    \item[1a.2]$D=-A$;
    \item[1a.3]$F=\emptyset$;
    \item[1a.4]$|E|\leq 1$;
    \item[1a.5]$E=\emptyset$.
\end{itemize}

\underline{Proof of 1a.1}.
We know that $-\pi(p)\notin B(1,\pi(p))$. This implies that $-\pi(p)\in A$, therefore $-\pi(p)\in D\cup A\cup F$. We consider $\pi(b)\in -B$, we see that $b\geq p$, this implies that $-b\leq -p$, therefore $-\pi(b)\notin E, B(1,\pi(n))$. This implies that $B=\emptyset$.

\underline{Proof of 1a.2}. We know that $-\pi(p)$ and $\pi(p)$ are not adjacent, so $-\pi(p)$ must be adjacent to $\pi(n)$ and therefore be an element of $A$. 
From $B=\emptyset$ and $|B\cup D|\geq 2$ we obtain that $|D|\geq 2$, in particular $\pi(p_-)\in D$ and $\pi(p)$ is not adjacent to any element in $-D$ implying $-D\subset A$. We look now at $-\pi(p)$, this is adjacent to only one element in $A$, this element is $-\pi(p_-)$. Assume that there is a $\pi(a)\in A\setminus -D$, we have that $\pi(a)$ is not adjacent to $-\pi(p)$, then $-\pi(a)$ is not adjacent to $\pi(p)$ nor to $\pi(n)$, obtaining a contradiction.
 
 \underline{Proof of 1a.3}. 
From the fact that $\pi(p_-)$ is in $D$ and that the elements in $D\cup F\cup A \cup E$ are in ascending order it follows that $\pi(p)$ is not adjacent to any element in $-F$. Because we are assuming $G=\emptyset$ we obtain that $\pi(n)$ is not adjacent to any element in $-F$, it follows that $F=\emptyset$. 
 
\underline{Proof of 1a.4}. By the monotonicity of the elements in $E$ follows that $\pi(p)$ is adjacent to at most one elements in $-E$, this implies that $|E|\leq 1$.

\underline{Proof of 1a.5}.We study the elements in $C$: we know that $\pi(p_+)\in C\neq \emptyset$ moreover we want all the elements in $-C$ to be adjacent with $\pi(p)$. This implies that for any $\pi(c)\in C$, $0<-p_-<c<-p$ except for the biggest $c$ that may be bigger than $-p$. Let now $\pi(e)\in E$, we see that the elements in $B(1,\pi(e))\cup B(1,\pi(n))$ are at most two: $-\pi(p)$ and $-\pi(c)\in C$ with $c$ maximal. 
We already noticed that $|B(1,\pi(e))\cap B(1,\pi(n))|=2$ implies that every element is adjacent to either $\pi(e)$ either $\pi(n).$
We remark that $-\pi(e)$ is not adjacent to both $\pi(e)$ and $\pi(n)$ and conclude that $E=\emptyset$.

\vspace{3pt}

We are in the situation where the only non-empty subsets of $B(1,\pi(n))$ are $A,D$ and $C$. We are studying case $n$ odd and we assumed $|B(1,c)|=n+1$, because $A=-D$ we must have that $|C|\geq 2$. Let $\pi(c)\in C$ be the one with the maximal $c$, we study $B(1,\pi(c))\cap B(1,\pi(n))$ and see that there are exactly two elements in it: $-\pi(p)$ and $\pi(c')$ which is the one with the maximal $c'$ in $C\setminus \{\pi(c)\}$. We also know that $-\pi(c)$
is not adjacent to $\pi(c)$ nor to $\pi(n)$. 
This implies that $|C|\leq 1$ which is absurd.

\textbf{Subcase 1.b)}. Now we assume that $|B(1,\pi(p))\cap B(1,\pi(n))|=3$. This can happen only if there are exactly two elements in $D$ (namely $\pi(p)$ and $\pi(p_-)$), no elements in $F$ and at least one element in $A$ ($-\pi(p_-)$). Notice that in this subcase we are allowing at most one element not to be adjacent to both $\pi(n)$ and $\pi(p)$. We claim:
\begin{itemize}
    \item[1b.1] $|B|\leq 1$;
    \item[1b.2] $|E|\leq 2$;
    \item[1b.3] $-\pi(p)\in A \rightarrow A=-D$;
    \item[1b.4] $A\neq -D\Rightarrow |E|=\emptyset$;
    \item[1b.5] $A=-D\Rightarrow |E|\leq 1$;
    \item[1b.6] $|A|=2$;
    \item[1b.7] $|C|\leq 1$;
    \item[1b.8] $E=\emptyset$.
\end{itemize}

\underline{Proof of 1b.1}. We see that the elements in $-B$ are not adjacent to $\pi(n)$. It is sufficient to notice that for any $\pi(b)\in B$, $-b<-p<n$ and $-\pi(b)<-\pi(p)<\pi(n)$. There can be at most one element not adjacent to $\pi(1)$ and $\pi(n)$, therefore $|-B|\leq 1$.

\underline{Proof of 1b.2}.We consider the reasoning did in the previous subcase for $1a.4$  to study $|E|\leq 1$. It hold in the present subcase as well except we now allow the existence of one element not adjacent to $\pi(p)$ and $\pi(n)$.

\underline{Proof of 1b.3}.
Assume that $-\pi(p)\in A$, we have that this is its rightmost element. Conversely, $-\pi(p_-)$ is the leftmost element in $A$. We have that if there was an element in $A\setminus \{-\pi(p_-),\pi(p)\}$, then $-\pi(p_-)$ and $-\pi(p)$ wouldn't be adjacent, by the monotonicity of the elements in $A$. We conclude that $A=\{-\pi(p),-\pi(p_-)\}=-D$ because $\pi(p)$ and $\pi(p_-)$ are adjacent in $\Gamma(\pi)$ therefore also $-\pi(p)$ and $-\pi(p_-)$ must be. 

\underline{Proof of 1b.4}. Suppose that $-\pi(p)\notin A$. This implies that there are some elements $\pi(a_1)<\ldots<\pi(a_k)$ in $A$ such that $\pi(p)<\pi(a_j)<0$ for $j=1,\ldots,k$ and such that some of these $a_j$ satisfy $p<a_j$. This last condition makes $\pi(n)$ and $-\pi(p)$ not adjacent in $\Gamma(\pi)$. We have that $-\pi(p)$ is not adjacent to $\pi(p)$ nor to $\pi(n)$, meaning that all the other elements must be adjacent to either $\pi(p)$ either $\pi(n)$. We consider the elements in $-C$, we want them to be adjacent to $\pi(p)$, in particular this means that given $C=\{\pi(c_1)<\pi(c_2)<\ldots<\pi(c_t))\}$ then $\pi(c_t)<\pi(a_k)$. We consider now $\pi(e)$ the rightmost element in $E$, we see that $B(1,\pi(e))\cap B(1,\pi(n))$ has at most one element which is the rightmost in $A\cup E\setminus \{\pi(e)\}$. We conclude that $\pi(e)$ is adjacent to at most $n$ vertices in $\Gamma(\pi)$ and so $E=\emptyset$.

\underline{Proof of 1b.5}. Assume that there are exactly two elements in $E$: $\pi(e_1)<\pi(e_2)$. To prove this claim we consider the elements in $-C$ and notice that these must be all adjacent to $\pi(p)$ except at most one. Therefore all the elements $\{\pi(c_1)<\ldots<\pi(c_t)\}$ in $C$ must satisfy $\pi(c_1)<\ldots<\pi(c_{t-1})<\pi(e_1)$, except $\pi(c_t)$ that may be bigger that $\pi(e_1)$. We consider the elements adjacent to $\pi(e_2)$ and $\pi(n)$ and see that these are at most two. This means that there can't be an element in $(B(1,\pi(e_2))\cup B(1,\pi(n))^c$. We conclude noticing that $-\pi(e_2)$ isn't adjacent to $\pi(e_2)$ and $\pi(n)$.

\underline{Proof of 1b.6}. The claim is trivial if $A=-D$. Assuming $A\neq -D$ we still have that the  leftmost element in $A$ is $-\pi(p_-)$ and all the other elements in $A$ must be at the right of $-\pi(p)$. Studying $A\neq -D \Rightarrow E=\emptyset$
 we have described the positions of the elements in $A$ and the elements in $C$. Let $\pi(a_j)$ be the rightmost element in $A$, we have that $B(1,\pi(a_j))\cap B(1,\pi(n))$ has at most two elements: one in $C$, one in $A$. We notice now that $-\pi(a_j)$ is not adjacent to $\pi(n)$. Because $-\pi(a_j)\in (B(1,\pi(a_j))\cup B(1,\pi(n)))^c$ we conclude.
 
\underline{Proof of 1b.7}. Assume that $|C|\geq 2$ and let $\pi(c)$ be its rightmost element. We study $B(1,\pi(c))\cap B(1,\pi(n))$, we see that there are at most two elements from $A\cup E$ and at most one element from $C\setminus\{\pi(c)\}$. We conclude noticing that $-\pi(c)$ is not adjacent to $\pi(c)$ and $\pi(n)$.

\underline{Proof of 1b.8}. This can be proved following the reasoning of \underline{$E=\emptyset$} in subcase $1.a$.

We are studying the case of $n\geq 5$, this implies $|B(1,\pi(n))|\geq 6$. So far we have proved that in this subcase
\[
|A|+|D|+|B|+|C|\leq 6.
\]
We assume therefore $|B|=|C|=1$ and $n=5$. We notice that $\pi(c),-\pi(c)$ $,-\pi(b)$ $,-\pi(n)$
and $\pi(i)$ cannot be adjacent to $\pi(c)$, this implies that $|B(1,\pi(c))|\leq n$ which is a contradiction that proves Subcase $1.b$ . 

\textbf{Subcase 1.c)}. In this case we are assuming that $B(1,\pi(p))\cap B(1,\pi(n))=\{\pi(p_-),\pi(p_+)\}$, in particular, this means that there is no $\pi(p_+)$ and so $B\cup C=\emptyset$. Also, we have that $F=\emptyset$, therefore $B(1,\pi(n))=D\cup A\cup E$. Let $\pi(e)$ the rightmost element in $E$, then $B(1,\pi(e))$ and $B(1,\pi(n))$ have at most one element in common, this is one element in $A\cup E\setminus \{\pi(e)\}$. Because $|B(1,\pi(n))|>n$, we obtain that $|B(1,\pi(e))|<n$. This briefly concludes the current subcase.

\textbf{2nd Case}
We are assuming that there is at least one $\pi(k)\in B(1,\pi(n))$ such that $\pi(k)<0$ and $k<0$ and no elements $\pi(j)\in B(1,\pi(n))$ such that $0<\pi(j)<\pi(n)$ and $j>0$. 

\underline{Claim: $\pi(k)<0$ and $k<0\Rightarrow  \pi(k)<-\pi(n)$}. If we have a $\pi(k)\in B(1,\pi(n))$ such that $-\pi(n)<\pi(k)<0$ and $k<0$ then $-\pi(k)$ would satisfy  $0<\pi(k)<\pi(n)$ and such that $k>0$. Thus we don't want $-\pi(k)$ to be adjacent to $\pi(n)$, we notice that this is possible only if there is a $\pi(j)$ in $B(1,\pi(n))$ such that $0<\pi(j)<\pi(n)$ and $j>0$, but this goes against our assumptions for this second Case. 

Therefore we have that there are no elements such that $-\pi(n)<\pi(k)<0$, we notice that this is possible if and only if $\pi(n)=1$.
%We consider $\pi(i)$, $\pi(i_+)$ and $\pi(i_-)$ as defined previously defined. From the fact that $\pi(n)=1$ we have that $\pi(i)\in B$, $\pi(i_+)\in B\cup C$ and $\pi(i_-)\in G$.

We divide the set of elements adjacent to $\pi(n)$ in two sets:
\begin{itemize}
\item we let $U$ be the set of elements $\pi(u)$ adjacent to $\pi(n)$ such that $\pi(u)>\pi(n)$;
\item we let $L$ be the set of elements $\pi(l)$ adjacent to $\pi(n)$ such that $\pi(l)<-\pi(n)$;
\end{itemize}

The elements in $U$ must be in ascending order while the elements in $L$ must be in descending order (see Picture \ref{picturedn2}). In particular the set $L$ and the set $-U$ have at most one element in common, meaning that $|B(1,\pi(n))\cap B(1,-\pi(n))|\leq 2$. Now notice that $\pi(n),-\pi(n)\notin B(1,\pi(n))$, this implies that  the at least $n+1$ elements in 
$B(1,\pi(n))$ range in the set $\{\pm\pi(1),\ldots,\pm\pi(n-1)\}$. In particular 
there are at least two elements $\pi(k)>0$ in $B(1,\pi(n))$ such that $-\pi(k)\in B(1,\pi(n))$. This would imply that $|B(1,\pi(n))\cap B(1,-\pi(n))|\geq 4$, we reached therefore a contradiction.

\textbf{3rd Case}.
In this case we assume that there are some elements $\{\pi(f_i)\}_I\subset B(1,\pi(n))$ such that
$0<\pi(f_i)<\pi(n)$ and $f_i>0$ and at some elements $\{\pi(g_j)\}_J\subset B(1,\pi(n))$ that satisfy $\pi(g_j)<0$ and $g_j<0$ with $I$ and $J$ non-empty.
 Because $\pi(g_1)\in B(1,\pi(n))$ we have that there are no elements $\pi(h)\in B(1,\pi(n))$ such that $g_1<h<n$, $\pi(g_1)<\pi(h)<\pi(n)$ and $\pi(h)\neq -\pi(g_1)$. Because both $g_1$ and $\pi(g_1)$ are negative we conclude that $|I|=1$ and that $\pi(f_1)=-\pi(g_1)$. On the other hand assume that $|J|>1$ and take $\pi(g_j)$ and $j\neq 1$, then this is not adjacent to $\pi(n)$ because $\pi(g_j)<-\pi(g_1)<\pi(n)$ and $k'<-g<n$. We conclude that $|I|=|J|=1$ and we redefine $\pi(f_1)=:\pi(f)$
and $\pi(g_1)=:\pi(g)$.
There are at least $n+1$ elements in $B(1,\pi(n))$, in particular there are at least two $\pi(i)$ with $i<0$. One of these is $\pi(g)$, then there is at least one $\pi(i)\in B(1,\pi(n))$, $i<0$ and $\pi(i)>0$, we call $\pi(p)$  the left most of the elements with this property. We consider now the definitions of the sets $A$, $B$, $C$, $D$, $E$, $F$ and $G$ from the 1st case and adapt them to this 3rd case substituting the role of $\pi(p)$ (see Picture \ref{picturedn3}).

if $|B(1,\pi(p))\cap B(1,\pi(n))|<2$ we would have $|B(1,\pi(p))|\leq n$ proving the statement, we assume therefore $|B(1,\pi(p))\cap B(1,\pi(n))|\geq2$. We divide the study of this case in $3$ subcases:
\begin{itemize}
\item[3.a] In this subcase we consider that $\pi(p)$ is adjacent to exactly two elements in $B(1,\pi(n))$, these are $\pi(p_-)$, which is the biggest element in $D\setminus{\pi(p)}$ and the second is $\pi(p_+)\in B\cup C$; 
\item[3.b] in this subcase we consider that $\pi(p)$ is adjacent to exactly 3 elements in $B(1,\pi(n))$, these are $\pi(f)$, $\pi(g)$ and $\pi(p_+)\in B\cup C$;
\item[3.c] in this subcase we consider that $\pi(p)$ is adjacent to exactly two elements in $B(1,\pi(n))$: $\pi(f)$ and $\pi(g)$.
\end{itemize}

\textbf{Subcase 3.a}. We take $\pi(p)$ and assume that $\pi(f),\pi(g)\notin B(1,\pi(p))$.
We claim:
\begin{itemize}
    \item[3a.1] $A=-D$;
    \item[3a.2] $|E|\leq 1$;
    \item[3a.3] $B=\emptyset$;
    \item[3a.4] $|C|=1$; 

\end{itemize}
This subcase is similar to subcase $1.a$, and following the same reasoning done in that case we see that every element must be adjacent either with $\pi(i)$ or with $\pi(n)$. We obtain $3a.1$, $3a.2$ and $3a.3$ following the proof of $1a.2$, $1a.4$ and $1a.1$ respectively.

\underline{Proof of 3a.4} From the fact that $B=\emptyset$ follows that $|C|\geq 1$ because $\pi(p_+)\in C$. We want the elements in $-C$ to be adjacent to $\pi(p)$. This implies that $-\pi(p_-)<\pi(c)<-\pi(p)$ for all the $\pi(c)\in C$ except from the rightmost which may be bigger than $-\pi(p)$. We take $\pi(c)$ as the rightmost element in $C$ and study $B(1,\pi(c))\cap B(1,\pi(n))$. We see that it has cardinality at most $4$: namely $\pi(c)$ is adjacent to at most two elements in $A\cup E$, at most one element in $C\setminus {\pi(c)}$ and may be adjacent to $\pi(f)$. Notice that if $|C|>1$,  $\pi(c)$ is not adjacent to $\pi(g)$.
 Therefore there must be at most two elements not adjacent to both $\pi(c)$ and $\pi(n)$. The  fact that $\pi(c)$ is adjacent to another element in $C$ implies that is is not adjacent to any element $\pi(k)$ with $k<0$. In particular neither $\pi(n)$ nor $\pi(c)$ are adjacent to any element in $-C$ nor to $-\pi(n)$ . This implies that $|-C|\leq 1$ in particular $C=\{\pi(p_+)\}$. 
 
  The number of elements in $D\cup F\cup G\cup A$ is even and $|B(1,\pi(n))|$ is even as well,  
 it follows that $|E|=|C|=1$. We study $B(1,\pi(e))$, $\{\pi(e)\}= E$, we have that its intersection with $B(1,\pi(n))$ has cardinality $2$: it contains $-\pi(p)$ and $\pi(c)\in C$. We also notice that $-\pi(e)$ is not adjacent to $E$ and conclude that $B(1,\pi(e))\leq n$.

\textbf{Subcase 3.b}. In this subcase we consider  $B(1,\pi(i))\cap B(1,\pi(n))\supset F\cup G\cup \{\pi(p_+)\}$. In particular we have that $|D|=1$
. There must be at most one vertex in $\Gamma(\pi)$ which is not adjacent to $\pi(n)$ and $\pi(p)$. This situation is similar to the one presented in case $1.b$ and following the same reasoning we see that ${|B|\leq 1}$ (see 1b.1) and ${|E|\leq 2}$ (see 1b.2). Furthermore we claim the following:
\begin{itemize}
    \item[3b.1] $-\pi(p)\in A\Rightarrow A=-D$; 
    \item[3b.2] $|E|\leq 1, |A|= 1$;
    \item[3b.3] $|C|\leq 2$;
    \item[3b.4] $E=\emptyset$;
    \item[3b.5] $|C|\leq 1$. 
\end{itemize}

\underline{Proof of 3b.1}.
Assume $-\pi(p)\in A$, we want to see that there are no elements $\pi(k)$ such that $0<k<-p$ and $0<\pi(k)<-\pi(p)$. We notice that this is a consequence of assuming that $-\pi(p)$ is adjacent to $\pi(n)$, $\pi(g)$ is adjacent to $\pi(n)$ and that $-\pi(p)$ is adjacent to $\pi(g)$. 

\underline{Proof of 3b.2}. We consider again subcase $1.b$. We see that the proofs of $A=-D$ $\Rightarrow |E|$ $=0$ (see 1b.4) and $A\neq D \Rightarrow |E|\leq 1$ (see 1b.5)
 hold in the present subcase as well. The proof of $|A|=2$ (see 1b.6) can then be adapted to prove that $|A|=1$.
 
\underline{Proof of 3b.3}. We study the position of the elements in $C$. These must all be adjacent to $-\pi(p)$ except at most one. We consider the set $S=\pi^{-1}(F\cup A \cup E\setminus \{-\pi(p)\})$ and consider $s$ as the biggest element in $S$ smaller than $-p$ and $s'$ as the smallest element in $S$ such that $s'>-p$. In case $s$ or $s'$ don't exist we can take $s=0$ and $s'=n$. All the elements $\pi(c)\in C$ except one satisfy $s<c<s'$. Consider now the rightmost element in $C$, say $\pi(c)$ we want to study $B(1,\pi(c))\cap B(1,\pi(n))$. $\pi(c)$ is adjacent to at most one element in $C\setminus \{\pi(c)\}$, two elements in  $A\cup E$ and may be adjacent to $\pi(f)$. This means that there are at most two vertices of $\Gamma(\pi)$ not adjacent to $\pi(c)$ and $\pi(n)$. We notice that $-C$ is not adjacent to these two points and conclude that $|-C|\leq 2$.    

\underline{Proof of 3b.4} It is sufficient to notice that $-\pi(e)$ is not adjacent to $\pi(e)$ nor to $\pi(n)$ and $|B(1,\pi(e))\cap B(1,\pi(n))|=2$. 

\underline{Proof of 3b.5}. Assume that $C=\{\pi(c_1),\pi(c_2)\}$ with $c_1<c_2$. We study $B(1,\pi(c_2))\cap B(1,\pi(n))$ and see that this intersection contains at most three elements: the one in $A$, $\pi(f)$ and $\pi(c_1)$. Notice now that the elements in $-C$ are in $(B(1,\pi(c_2))\cap B(1,\pi(n)))^c$, the latter having cardinality at most $2$.

Now we compute the number of elements that may be in $B(1,\pi(n))$:
\[
|B(1,\pi(n))|\leq \underbrace{2}_{A,D}+\underbrace{2}_{G,F}+\underbrace{1}_{C}+\underbrace{1}_{B}=6.
\]
We know that we are studying the case of $n$ odd number and $|B(1,\pi(n))|>n$, so we only have to study the case of $n=5$. We notice that $\pi(c)\in C$ is not adjacent to the following elements:$\pi(c)$,$-\pi(c)$, $-\pi(b)$, $\pi(p)$
and $-\pi(n)$. We conclude that the hypothesis holds also for this case.

\textbf{Subcase 3.c}. In this subcase we are assuming $B\cup C= \emptyset$.
Notice that in this case $|D|=|F|=|G|=1$, it follows that there are at least $3$ elements in $A\cup E$.
We consider the rightmost element in $A\cup E$, say $\pi(e)$ and  notice that it is adjacent to at most one element in $B(1,\pi(n))$, namely the rightmost in $A\cup E\setminus \{\pi(e)\}$, this implies that $|B(1,\pi(e))|\leq n$ in contradiction wit the hypothesis.

\vspace{1pt}

The three cases above prove the statement for $\pi(n)>0$, yet we can notice that during the proof we didn't use the fact that the $i$s in $\{1,\ldots,n\}$ such that $\pi(i)<0$ are even. Therefore the whole reasoning applies for the case $\pi(n)<0$. 

\end{proof}

Now we are ready for the main result:
\begin{teo}\label{degdn}
The following identity holds:
\[
d_{\max}(H(D_n))=\lfloor \frac{n^2}{2}\rfloor+n-1.
\]
\end{teo}
\begin{proof}
We prove the $'\leq'$ part of the statement by induction. We first take the case $n=3$, we know that $B_3$ with the strong Bruhat order is isomorphic to $A_3=S_4$ with the strong Bruhat order. By \cite{adinroich} we know that the maximal degree for $H(S_4)$ is $\lfloor (\frac{4}{2})^2\rfloor+4-2=6=\lfloor \frac{3^2}{2}\rfloor+3-1$; this concludes the base step of the induction.

Assume now that the hypothesis holds for $n-1=2m$, an even number, and let's prove it for $n$. We consider $\pi\in D_{n}$ and take the graph $\Gamma(\pi)$. Let $a$ a vertex in $\Gamma(\pi)$ with degree at most $n$. We consider $\bar{\pi}\in D_{n-1}$ as the element obtained by erasing $a$ from the window notation of $\pi$ and decreasing by
one the values greater than one. The total number of edges in $\Gamma(\bar{\pi})$ (resp. $\Gamma(\pi)$) is equal to the degree of $\bar{\pi}$ in $H(D_{n-1})$, the following bound on the number $d(\pi)$ holds:
\begin{align*}
d(\pi)&=|\{\text{number of edges in } \Gamma(\pi)\}|\\
&\leq |\{\text{number of edges in } \Gamma(\bar{\pi})\}+n\\
&\leq \lfloor \frac{2m^2}{2}\rfloor+2m-1+n.
\end{align*}
We obtain:
\begin{align*}
d(\pi)&\leq \lfloor \frac{(n-1)^2}{2}\rfloor+(n-1)-1+n \\
      & \leq \lfloor \frac{n^2-2n+1}{2}\rfloor+(n-1)-1+n\\
    & \leq \lfloor \frac{n^2+1}{2}\rfloor +n-2
      \underbrace{=}_{n \text{ is odd}}  \lfloor \frac{n^2}{2}\rfloor+n-1.
\end{align*}
This concludes the proof in the case of $n$ odd. Now we assume that $n-1$ is odd and prove the statement for $n$. We take $\pi\in D_{n}$ and let $a$ be an element in $\Gamma(\pi)$ with degree smaller than or equal to $n+1$. We define $\bar{\pi}\in D_{n-1}$ as we did in the previous case and we bound the edges of $\Gamma(\pi)$ as follows:
\[
d(\pi)\leq |\{\text{number of edges in }\Gamma(\bar{\pi})\}|+(n+1)\leq \lfloor \frac{(n-1)^2}{2}\rfloor+(n-1)-1+n+1.
\]
We obtain
\begin{align*}
d(\pi)&\leq \lfloor \frac{(n-1)^2}{2}\rfloor+(n-1)-1+n+1 \\
      & \leq \lfloor \frac{n^2-2n+1}{2}\rfloor+(n-1)-1+n+1\\
    & \leq \lfloor \frac{n^2+1}{2}\rfloor +n-1
      \underbrace{=}_{n \text{ is even}}  \lfloor \frac{n^2}{2}\rfloor+n-1.
\end{align*}
We have obtained that $d(\pi)\leq \lfloor \frac{n^2}{2}\rfloor+n-1 $ for any $\pi\in D_n$, to prove that this value is reached by an element we consider the following:
\[
[1,\ldots,m,-n,\ldots,-(m+1)]\quad \text{with }m=\lfloor \frac{n}{2}\rfloor.
\]
An easy computation shows that this element belongs to exactly $ \lfloor \frac{n^2}{2}\rfloor+n-1$ edges in $H(D_n)$.
\end{proof}

\section*{Acknowledgement}
The author would like to thank Francesco Brenti for suggesting her this problem and for many useful discussions, furthermore she acknowledges the MIUR Excellence Department Project awarded to the Department of Mathematics, University of Rome Tor Vergata, CUP E83C18000100006.

\printbibliography
\end{document}